%% file: varlatest2.tex
\documentclass[11pt]{amsart}
\usepackage{amssymb}
\usepackage{verbatim}
\usepackage{hyperref}
\usepackage{color}

\usepackage{scalerel}
\newcommand\vwidehat[1]{\arraycolsep=0pt\relax%
\begin{array}{c}
\stretchto{
  \scaleto{
    \scalerel*[\widthof{\ensuremath{#1}}]{\kern-.5pt\bigwedge\kern-.5pt}
    {\rule[-\textheight/2]{1ex}{\textheight}} 
  }{\textheight} %
}{0.5ex}\\           
#1\\                 
\rule{-1ex}{0ex}
\end{array}
}

\usepackage{mathabx}

\newtheorem{theorem}{Theorem}    
\newtheorem{proposition}[theorem]{Proposition}

\newtheorem{corollary}[theorem]{Corollary}
\newtheorem{lemma}[theorem]{Lemma}
\newtheorem{sublemma}[theorem]{Sublemma}

\theoremstyle{definition}
\newtheorem{definition}{Definition}
\newtheorem{notation}[definition]{Notation}
\newtheorem{remark}[definition]{Remark}

\numberwithin{theorem}{section}
\numberwithin{definition}{section}
\numberwithin{equation}{section}

\title[A constrained optimization problem]{A constrained optimization problem for\\  the Fourier transform: Quantitative analysis}

\author{Dominique Maldague}
\address{
        Dominique Maldague\\
        Department of Mathematics\\
        University of California \\
        Berkeley, CA 94720-3840, USA}
\email{dmaldague@berkeley.edu}

\date{June 1, 2017.}

\input{dfmacros.tex}

\begin{document}

\setcounter{tocdepth}{1}

\begin{abstract} 
Among functions $f$ majorized by indicator functions $1_E$, which functions have maximal ratio $\|\widehat{f}\|_q/|E|^{1/p}$? We establish a quantitative answer to this question for exponents $q$ sufficiently close to even integers, building on previous work proving the existence of such maximizers.
\end{abstract}

\maketitle

\tableofcontents

\section{Introduction}

Define the Fourier transform as $ \mc{F}(f)(\xi)=\widehat{f}(\xi)=\int_{\R^d}e^{- 2\pi i x\cdot \xi}f(x)dx$ for a function $f:\R^d\to\C$. The Fourier transform is a contraction from 
$L^1(\R^d)$ to $L^\infty(\R^d)$ and is unitary on $L^2(\R^d)$. Interpolation gives the Hausdorff-Young inequality $\|\widehat{f}\|_q\le \|f\|_p$ where $p\in (1,2)$, $1=\frac{1}{p}+\frac{1}{q}$. In \cite{beckner}, Beckner proved the sharp Hausdorff-Young inequality 

\begin{align} \|\widehat{f}\|_q\le {\bf{C}}_q^d\|f\|_p \label{maineq} \end{align}
where ${\bf{C}}_q=p^{1/2p}q^{-1/2q}$. In 1990, Lieb proved that Gaussians are the \emph{only} maximizers of (\ref{maineq}), meaning that $\|\widehat{f}\|_q/\|f\|_p={\bf{C}}_q^d$ if and only if $f=c\exp(-Q(x,x)+v\cdot x)$ where $Q$ is a positive definite real quadratic form, $v\in\C^d$ and $c\in\C$. In 2014, Christ established a sharpened Hausdorff-Young inequality by bounding $\|\widehat{f}\|_q - {\bf{C}}_p^d \|f\|_p$ by a negative multiple of an $L^p$ distance function squared of $f$ to the Gaussians.

In \cite{c2}, Christ proved the existence of maximizers for the ratio $\|\widehat{1_E}\|_q/|E|^{1/p}$ where $E\subset\R^d$ is a positive Lebesgue measure set. For $d\ge 1$, $q\in(2,\infty)$, and $p=q'$, define 
\[ {\bf{A}}_{q,d}:=\sup_{|E|<\infty}\frac{\|\widehat{1_{E}}\|_q}{|E|^{1/p}}\]
where the supremum is taken over Lebesgue measurable subsets of $\R^d$ of finite measure. Building on the work of Burchard in \cite{burchard}, Christ identified maximizing  sets to be ellipsoids for exponents $q\ge 4$ sufficiently close to even integers \cite{c2}.

Another variant of the Hausdorff-Young inequality replaces indicator functions by bounded multiples and modifies the functional as follows.
For $d\ge 1$, $q\in(2,\infty)$, and $p=q'$, we consider the inequality

\begin{equation}\label{eqn:main}
\|\widehat{f}\|_q\le {\bf{B}}_{q,d}|E|^{1/p}
\end{equation}
and define the quantities 
\begin{align}
\Psi_q(E):=\sup_{|f|\prec E} \frac{\|\widehat{f}\|_{q}}{\|1_E\|_p} \label{eq2}\\
{\bf{B}}_{q,d}:=\sup_{E}\Psi_q(E) \label{thisone}
\end{align}
where $|f|\prec E$ means $|f|\le 1_E$ and the supremum is taken over all Lebesgue measurable sets $E\subset\R^d$ with positive, finite Lebesgue measures. This quantity ${\bf{B}}_{q,d}$ is less than ${\bf{C}}_p^d$ by their definitions and \cite{c2}. The supremum (\ref{thisone}) is equal to 

\[   \sup_{f\in L(p,1)}\frac{\|\widehat{f}\|_q}{\|f\|_{\mc{L}}}\quad\quad\text{where}\quad\quad \|f\|_{\mc{L}}=\inf\{\|a\|_{\ell^1}\,:|f|=\sum_n a_n|E_n|^{-1/p}1_{E_n},\, a_n>0,|E_n|<\infty\} .\]

See a discussion of this equivalence in \textsection 2 of \cite{me!}. Lorentz spaces are a result of real interpolation between $L^p$ spaces. Since the quasinorm $\|\cdot\|_\mc{L}$ induces the standard topology on the Lorentz space $L(p,1)$, this is a natural quantity to study. The existence of $f\prec E$ such that ${\bf{B}}_{q,d}=\|\widehat{f}\|_q/|E|^{1/p}$ was established by this author in \cite{me!} using additive combinatorial techniques from Christ \cite{c1,c2}.

In this paper, we prove ${\bf{B}}_{q,d}={\bf{A}}_{q,d}$, identify the maximizers, and prove a quantitative stability result for the inequality 
\begin{align}
\|\widehat{f}\|_q\le {\bf{B}}_{q,d}|E|^{1/p} \label{eq:main}
\end{align}
when $q$ is near an even integer $m\ge 4$. We refer the reader to \cite{c1} for a discussion of quantitative stability results in analysis. We define some notation in order to state our main theorem. Let $\mathfrak{E}$ denote the set of ellipsoids in $\R^d$. Let $A\Delta B$ denote the symmetric difference $(A\setminus B)\cup(B\setminus A)$. For any Lebesgue measurable subset $E\subset\R^d$ with $|E|\in(0,\infty)$, define 
\begin{align}
\text{dist}(E,\mathfrak{E}):= \inf_{\mc{E}\in\mathfrak{E}}\frac{|\mc{E}\Delta E|}{|E|}
\end{align}
where the $\inf$ is taken over all ellipsoids satisfying $|\mc{E}|=|E|$. Let $\mathfrak{L}$ denote the set of affine functions $L:\R^d\to\R$. For $e^{ig}\in L^2(E)$ with $g$ real-valued,  define
\begin{align}
\text{dist}_E(e^{ig},\mathfrak{L}):= \inf_{L\in\mathfrak{L}}\|e^{ig}-e^{iL}\|_{L^2(E)}. 
\end{align}

\begin{theorem}\label{mainthm} Let $d\ge 1$. For each even integer $m\in\{4,6,8,\ldots\}$ there exists $\delta(m)>0$ such that the following conclusions hold for all exponents satisfying $|q-m|\le \delta(m)$. Firstly, if $E\subset\R^d$ is a Lebesgue measurable set of finite measure, and $f,g$ are real-valued functions with $0\le f\le 1$, then 
\[ {\bf{B}}_{q,d}= \|\widehat{fe^{ig}1_E}\|_q/|E|^{1/p}\]
if and only if $fe^{ig}1_E= e^{iL}1_{\mc{E}}$, where $L\in\mathfrak{L}$ and $\mc{E}\in\mathfrak{E}$. Secondly, there exists $c_{q,d}>0$ such that for every set $E\subset{\R^d}$ with $|E|=1$, and all $f,g$ real-valued functions with  $0\le f\le 1$, 

\begin{align}\label{stability}  \|\widehat{fe^{ig}1_E}\|_q^q\le {\bf{B}}_{q,d}^q-c_{q,d}\left[ \|f-1\|_{L^1(E)}+\text{dist}_E(e^{ig},\mathfrak{L})^2+\text{dist}(E,\mathfrak{E})^2 \right] .    \end{align}
\end{theorem}

This theorem refines (\ref{eq:main}) by majorizing $\|\widehat{f}\|_q/|E|^{1/p}-{\bf{B}}_{q,d}$ by a negative function of a distance of $(f,E)$ to the set of maximizers (or extremizers) of (\ref{eq:main}). The optimality of the $L^1$ norm and exponent 1 in $\|f-1\|_{L^1(E)}$ as well as the $L^2$ norm and exponent $2$ in $\text{dist}_E(e^{ig},\mathfrak{L})^2$ from (\ref{stability}) is proved in Lemma \ref{optf} and \textsection\ref{opt22} respectively. The optimality of the exponent $2$ of $\text{dist}(E,\mathfrak{E})^2$ is addressed in \cite{c2}.

We rely on the hypothesis that $q$ is near an even integer $m\ge 4$ to identify maximizers. For $q\in (2,\infty)$ not near an even integer, it is not known which sets $E$ maximize $\|\widehat{1_E}\|_q/|E|^{1/p}$. Furthermore, if $q=2m$ for $\Z\ni m\ge 2$, then for any $|f|\le 1_E$ where $E$ is a Lebesgue measurable set and $|E|\in(0,\infty)$, we have the inequality
\begin{align} 
\|\widehat{f}\|_{q}^{q}=\|f*\cdots*f\|_2^2\le \|1_E*\cdots*1_E\|_2^2= \|\widehat{1_E}\|_{q}^{q} \label{useful} \end{align}
where the convolution products are $m$-fold. The failure of $\|\widehat{f}\|_q\le \|\widehat{|f|}\|_q$ for general $f\in L^{q'}$ was shown for $q=3$ by Hardy and Littlewood and for all other exponents not in $\{2,4,6,\ldots\}$ by Boas \cite{boas}. Thus it is not obvious that ${\bf{B}}_{q,d}={\bf{A}}_{q,d}$. By Theorem \ref{mainthm}, ellipsoids are among maximizers for certain exponents $q$, so the following corollary is an immediate consequence.

\begin{corollary} Let $d\ge 1$. For each even integer $m\in\{4,6,8,\ldots\}$ there exists $\delta(m)>0$ such that if $|q-m|\le \delta(m)$, then ${\bf{B}}_{q,d}={\bf{A}}_{q,d}$. 
\end{corollary}

The term $\|\widehat{1_E}\|_q^q$ in (\ref{useful}) was analyzed by Christ in \cite{c2}. We state his main result concerning $\|\widehat{1_E}\|_q^q$ in Theorem \ref{christmain} below. 

\begin{theorem} \cite{c2}\label{christmain} Let $d\ge 1$. For each even integer $m\in\{4,6,8,\ldots\}$ there exists $\delta(m)>0$ such that the following three conclusions hold for all exponents satisfying $|q-m|\le \delta(m)$. Let $p$ be the conjugate exponent to $q$. Firstly, 
\[ {\bf{A}}_{q,d}=\|\widehat{1_E}\|_q/|E|^{1/p}\quad\text{for any }E\in\mathfrak{E}. \]
Secondly, ellipsoids are the only extremizers; for any Lebesgue measurable set $E\subset{\R^d}$, with $0<|E|<\infty$, $\Phi_q(E)={\bf{A}}_{q,d}$ if and only if $E$ is an ellipsoid. Thirdly, there exists $\tilde{c}_{q,d}>0$ such that for every set $E\subset{\R^d}$ with $|E|=1$, 
\begin{align}  \|\widehat{1_E}\|_q^q\le {\bf{A}}_{q,d}^q-\tilde{c}_{q,d}\textrm{dist}(E,\mathfrak{E})^2. \label{christstability} \end{align}
\end{theorem}

Suppose that $q\ge 4$ is an even integer. It is immediate from (\ref{useful}) that ${\bf{A}}_{q,d}={\bf{B}}_{q,d}$. Since for $|f|\le 1_E$,  
\[\frac{\|\widehat{f}\|_q}{|E|^{1/p}}=\frac{\|\widehat{e^{iL}f\circ \p}\|_q}{|\p^{-1}(E)|^{1/p}}\]
for all affine transformations $\p:\R^d\to\R^d$ and $L\in \mathfrak{L}$, we can also say that functions of the form $e^{iL}1_{\mc{E}}$ where $L\in\mathfrak{L}$ and $\mc{E}\in\mathfrak{E}$ are among the extremizers for (\ref{eq:main}). Establishing (\ref{stability}) would then show that they are the only extremizers. Christ's result in (\ref{christstability}) plus the inequality in (\ref{useful}) will provide the starting point for our proof of (\ref{stability}).

A general approach to proving stability results like (\ref{stability}) for $4\le q\in2\N$ is as follows. 
\begin{definition} For $\delta$ a small positive constant, we say that $|f|\le 1_E$ is a $\delta$ near extremizer of (\ref{eq:main}), or just a near extremizer, if $(1-\delta){\bf{B}}_{q,d}^q|E|^{q/p}\le \|\widehat{f}\|_q^q$.
\end{definition}

If $|E|=1$, $f,g$ are real-valued functions with  $0\le f\le 1$, and $fe^{ig}1_E$ is NOT a  $\delta$ near extremizer of (\ref{eq:main}), then 
\begin{align*}  
\|\widehat{fe^{ig}1_E}\|_q^q&\le {\bf{B}}_{q,d}^q(1-\delta)\le{\bf{B}}_{q,d}^q- \frac{\delta}{9}{\bf{B}}_{q,d}^q\left[ \|f-1\|_{L^1(E)}+\text{dist}_E(e^{ig},\mathfrak{L})^2+\text{dist}(E,\mathfrak{E})^2\right] 
\end{align*} 
since $ \|f-1\|_{L^1(E)}\le 1$, $\text{dist}_E(e^{ig},\mathfrak{L})^2\le 4$, and $\text{dist}(E,\mathfrak{E})\le 4$. Thus in the case that $fe^{ig}1_E$ is not a $\delta$ near extremizer, (\ref{stability}) is trivially satisfied with $c_{q,d}=\frac{\delta}{9}{\bf{B}}_{q,d}^q$. 

Now assume that $fe^{ig}1_E$ is a $\delta$ near extremizer. From (\ref{useful}) and (\ref{christstability}) for $4\le q\in2\N$, we can immediately say that 
\[    \tilde{c}_{q,d}\text{dist}(E,\mathfrak{E})^2\le \delta{\bf{B}}_{q,d}. \]
By precomposing $fe^{ig}1_E$ with an appropriate affine transformation, we can assume that $|E\Delta\B|^2$ is bounded by a constant multiple of $\delta$. We work more to prove that $f$ must be close to $1$ and $e^{ig}$ close to $e^{iL}$ for some $L\in\mathfrak{L}$ in \textsection\ref{nearext}. 

Let $\B$ denote the $d$-dimensional unit ball. In the case that $fe^{ig-iL}1_E$ is close to $1_{\B}$, we will be able to control the error in a Taylor expansion of $\|\widehat{fe^{ig-iL}1_E}\|_q^q$ about $\|\widehat{1_{\B}}\|_q^q$ which is developed in  \textsection\ref{taylorexpansion}. To simplify the Taylor expansion analysis, we treat the special case of $E=\B$ for near-even integer exponents $q$ in \textsection\ref{genq}. 

In \textsection\ref{suff}, we generalize the previous discussion to $3\le q$ near even integers using the equicontinuity of the functional 
\[ q\mapsto \|\widehat{f}\|_q \]
on $q\in(2,\infty)$ where $|f|\le 1_E$, $E$ a Lebesgue measurable set with $|E|<\infty$. Finally, for real valued functions $f$ and $g$ with $0\le f\le 1$ and a Lebesgue measurable set $E\subset\R^d$ of finite measure, we prove (\ref{stability}) for near extremizers in three cases: (1) majority modulus $f$ variation, (2) majority support $E$ variation, and (3) mostly frequency $g$ variation, which we address in Proposition \ref{modulus}, Proposition \ref{support}, and Proposition \ref{freq} respectively.

This material is based upon work supported by the National Science Foundation Graduate Research Fellowship under Grant No. DGE 1106400.


\section{Equicontinuity of $q\mapsto \|\widehat{f1_E}\|_q$ for $|E|=1$.}

The following equicontinuity result for the optimal constant ${\bf{B}}_{q,d}$ as a function of $q$ will be used to make a perturbative argument generalizing bounds for one exponent to nearby exponents. 

\begin{lemma}\label{equicontinuity} Let $d\ge 1$ and $r\in(2,\infty)$. As $f1_E$ varies over all subsets satisfying $|E|=1$ and functions satisfying $|f|\le 1$, the functions $q\mapsto\|\widehat{f1_E}\|_q$ form an equicontinuous family of functions of $q$ on any compact subset of $(2,\infty)$. 
\end{lemma}

\begin{proof} This follows from the proof of Lemma 3.1 from \cite{c2} with $f1_E$ in place of $1_E$. 
\end{proof}

An immediate consequence of the equicontinuity lemma is the following corollary. 
\begin{corollary} \label{constantcont}For each mapping $d\ge1$, the mapping $(2,\infty)\ni q\mapsto {\bf{B}}_{q,d}\in R^+$ is continuous.
\end{corollary}

The following corollary and lemma will be used in \textsection\ref{suff} to outline the strategy of the proof of Theorem \ref{mainthm}.

\begin{corollary} \label{suffcor}Let $d\ge 1$ and $\overline{q}\ge 4$ be an even integer with conjugate exponent $\overline{p}$. Let $\delta>0$, $E$ be a Lebesgue measurable subset of $\R^d$, and $f:\R^d\to\C$ satisfy $|f|\le 1$. Let $q>2$ with conjugate exponent $p$. If 
\[ \|\widehat{f1_E}\|_q^q/|E|^{q/p}\ge {\bf{B}}^q_{q,d}-\delta, \]
then 
\[ \|\widehat{f1_E}\|_{\overline{q}}^{\overline{q}}/|E|^{\overline{q}/\overline{p}}\ge {\bf{B}}_{\overline{q},d}^{\overline{q}}-o_{q-\overline{q}}(1)-\delta \]
where $o_{q-\overline{q}}(1)$ is a function which tends to zero as $|q-\overline{q}|$ goes to zero. 
\end{corollary}
\begin{proof}Since $\Psi_q$ is invariant under dilations, it suffices to consider when $|E|=1$.  Then the conclusion follows from the preceding Lemma \ref{equicontinuity} and Corollary \ref{constantcont}. 

\end{proof}

The purpose of the following lemma is to confirm that Theorem \ref{mainthm} is trivial unless $\|\widehat{fe^{ig}1_E}\|_q^q$ is close to ${\bf{B}}_{q,d}^q$. 

\begin{lemma} \label{sufficient}Let $d\ge 1$ and $q\ge 2$ with conjugate exponent $p$. Let $0<\delta<1$, let $E\subset\R^d$ be a Lebesgue measurable set with $|E|=1$ and let $f,g$ be real-valued functions with $0\le f\le 1$. If 
\[ \|\widehat{fe^{ig}1_{E}}\|_q^q\le {\bf{B}}_{q,d}^q-\delta, \]
then 
\[ \|\widehat{fe^{ig}1_{E}}\|_q^q\le {\bf{B}}_{q,d}^q-\frac{\delta}{6}\left[\|f-1\|_{L^1(E)}+\inf_{L\in\mathfrak{L}}\|e^{ig}-e^{iL}\|_{L^2(E)}^2+\text{dist}(E,\mathfrak{E})^2\right] . \]
\end{lemma}

\begin{proof} It suffices to note the following inequalities.
\begin{align*}
    \|f-1\|_{L^1(E)}&\le 2|E|\le 2\\
    \inf_{L\in\mathfrak{L}}\|e^{ig}-e^{iL}\|_{L^2(E)}&\le \|e^{ig}-1\|_{L^2(E)}\le 2|E|^{1/2}\le 2\\
    \text{dist}(E,\mathfrak{E})&\le \frac{|E\Delta\lambda \B|}{|E|}\le2, 
\end{align*}
where we define $\lambda$ by $|\lambda\B|=1$, where $\B$ denotes the unit ball in $\R^d$.  

\end{proof}


\section{Structure of near-extremizers of the form $fe^{ig}1_E$ for $q=2m$.\label{nearext}}

In this section, let $f$ be a real valued function with $0\le f(x)\le 1$ a.e., let $g$ be a real valued function, and let $E\subset \R^d$ be a Lebesgue measurable set. Recall that for even integers $q$, we know that $\|\widehat{1_{\B}}\|_q/|\B|^{1/p}={\bf{B}}_{q,d}$. We carefully unpackage the structure of near-extremizers of (\ref{eq:main}) of the form $fe^{ig}1_{E}$ for even $q$. By proving that (possibly after composition with an affine function) $e^{ig}$ must be close to a multiple of a character and that $\|f-1\|_1$ and $|E\Delta \B|$ must be small, we guarantee that a Taylor expansion of $\|\widehat{fe^{ig}1_{E}}\|_q^q$ about $\|\widehat{1_{\B}}\|_q^q$ will have an error that we can control (see \textsection{\ref{taylorexpansion}}). 

Since $q$ is even, we can write  $\|\widehat{fe^{ig}1_{E}}\|_q^q$ as an $m$-fold convolution product using Plancherel's theorem:
\begin{align} 
\|\vwidehat{fe^{ig}1_{E}}&\|_{2m}^{2m}=\int_{E^{2m-1}} f(x_1)\cdots f(x_m)f(y_2)\cdots f(y_m)f(L(x,y))\times\label{star} \\
    &\cos(g(x_1)+\cdots +g(x_m)-g(y_2)-\cdots -g(y_m)-g(L(x,y)))1_{E}(L(x,y)) dxdy \nonumber 
\end{align}
where $x=(x_1,\ldots,x_m)\in\R^{md}$, $y=(y_2,\ldots,y_m)\in\R^{(m-1)d}$, and $L(x,y)=x_1+\cdots +x_m-y_2-\cdots y_m$. From this expression, it is clear that
\begin{align}
    \|\widehat{fe^{ig}1_{E}}\|_q&\le\|\widehat{1_{E}}\|_q \label{ob1}\\
    \|\widehat{fe^{ig}1_{E}}\|_q&\le\|\widehat{f1_{E}}\|_q\label{ob2}\\
    \|\widehat{fe^{ig}1_{E}}\|_q&\le\|\widehat{e^{ig}f1_{E}}\|_q. \label{ob3}
\end{align}

If $(1-\delta){\bf{B}}_{q,d}|E|^{1/p}\le \|\widehat{fe^{ig}1_{E}}\|_q$, then by (\ref{ob1}),
\[  (1-\delta){\bf{A}}_{q,d}|E|^{1/p}\le \|\widehat{1_{E}}\|_q
  \]
where ${\bf{A}}_{q,d}=\sup_{E}\frac{\|\widehat{1_E}\|_q}{|E|^{1/p}}$ and equals ${\bf{B}}_{q,d}$ since $q$ is even. By Christ's Theorem 2.6 in \cite{c2}, conclude that 
\[|T^{-1}(E)\Delta \B|\le 2\text{dist}(E,\mathfrak{E})\le O(\delta^{1/2})\]
where $T\in {\text{Aff}(\R^d)}$ is an affine automorphism of $\R^d$ and the big-O depends on dimension and is uniform for $q$ in a compact subset of $(3,\infty)$.

Replacing our near-extremizer $fe^{ig}1_{E}$ by $f\circ Te^{ig\circ T}1_E\circ T$, we may assume that $|E\Delta \B|\le O(\delta^{1/2})$. 

Define a measurable function $f_0:\R^d\to [0,1]$ by $f_0=f1_{E\cap \B}+1_{\B\setminus E}$. Note that 
\begin{align}  
\|\widehat{f1_E}\|_q&\le \|\widehat{f_01_{\B}}\|_q+\|\widehat{f1_E}-\widehat{f_01_{\B}}\|_q\le \|\widehat{f_01_{\B}}\|_q+\|f1_{E\setminus\B}-1_{\B\setminus E}\|_p\nonumber\\
& \le \|\widehat{f_01_{\B}}\|_q+|E\Delta\B|^{1/p}\le \|\widehat{f_01_{\B}}\|_q+O(\delta^{1/2p}). \label{ob4}
\end{align} 

In the following lemma, we consider $\|\widehat{f_01_{\B}}\|_q$.

\begin{lemma}\label{nearextf} Let $d\ge 1$ and let $q\ge 4$ be an even integer with conjugate exponent $p$. Then
\[ \|\widehat{f1_{\B}}\|_q^q \le \|\widehat{1_{\B}}\|_{q}^{q}-c\|f-1\|_{L^1(\B)} \]
where $c=\inf_{\B}K_q>0$. 
\end{lemma}

\begin{proof}

Letting $q=2m$, we have

\begin{align*} 
 \|\widehat{f1_{\B}}\|_{2m}^{2m} =&\|\widehat{1_{\B}}\|_{2m}^{2m}+\|\widehat{f1_{\B}}\|_{2m}^{2m}-\|\widehat{1_{\B}}\|_{2m}^{2m}\\
    &=\|\widehat{1_{\B}}\|_{2m}^{2m}+\langle f1_{\B}*\cdots *f1_{\B},f1_{\B}*\cdots *f1_{\B}\rangle-\langle 1_{\B}*\cdots *1_{\B},1_{\B}*\cdots *1_{\B}\rangle \\
&\le  \|\widehat{1_{\B}}\|_{2m}^{2m}+\langle f1_{\B}*1_{\B}\cdots *1_{\B},1_{\B}*\cdots *1_{\B}\rangle-\langle 1_{\B}*\cdots *1_{\B},1_{\B}*\cdots *1_{\B}\rangle \\
&= \|\widehat{1_{\B}}\|_{2m}^{2m}+\langle (f-1)1_{\B},K_q \rangle \\
&\le \|\widehat{1_{\B}}\|_{2m}^{2m}-c\|f-1\|_{L^1(\B)}
\end{align*}
where each convolution product has $m$ factors, $\widehat{K_q}=\widehat{1_{\B}}|^{q-2}\widehat{1_{\B}}$ and we used that  $K_q>0$ on $\B$.
\end{proof}

Combine (\ref{ob4}) with Lemma \ref{nearextf} to reason that if $fe^{ig}1_E$ is a near-extremizer and $|E\Delta \B|\le O(\delta^{1/2})$, then 
\[ \|f-1\|_{L^1(E)}\le \|f_0-1\|_{L^1(\B)}+2|E\Delta\B|\le O(\delta^{1/2p}).  \]
Now define $g_0:\R^d\to \R$ by $g_0=1_{E\cap \B}g$. Then 
\begin{align*}
\|\widehat{fe^{ig}1_E}\|_q&\le \|\widehat{e^{ig_0}1_{\B}}\|_q+\|e^{ig}1_E-e^{ig_0}1_{\B}\|_p+\|(f-1)1_E\|_p \\
    &\le \|\widehat{e^{ig_0}1_{\B}}\|_q+|E\Delta\B|^{1/p}+\|(f-1)1_E\|_1^{1/p}\\
    &=\|\widehat{e^{ig_0}1_{\B}}\|_q+ O(\delta^{1/2p^2}). 
\end{align*}
Thus it remains to understand the case in which $e^{ig_0}1_{\B}$ is a near extremizer. The naive approach used to understand $f1_{\B}$ in Lemma \ref{nearextf} breaks down when considering $e^{ig}1_{\B}$ since an expression with $g$ appears within the argument of the cosine in (\ref{star}). One tool we have at our disposal is the following Proposition 8.2 of Christ from \cite{c3}, stated here for the reader's convenience.

\begin{lemma}\label{younglem} For each dimension $d\ge 1$ there exists a constant $K<\infty$ with the following property. Let $B\subset\R^d$ be a ball centerof positive radius, and let $\eta\in(0,\frac{1}{2}]$, and let $\delta>0$ be sufficiently small. For $j\in\{1,2,3\}$, let $f_j:2B\to\C$ be Lebesgue measurable
functions that vanish only on sets of Lebesgue measure zero. Suppose that 
\[ |\{(x, y)\in B^2: |f_1(x)f_2(y)f_3(x+y)^{-1}-1| > \eta\}|<\delta|B|^2. \]
Then for each index $j$ there exists an affine function $L_jn : \R^d\to\C$ such that
\[|\{x \in B : |f_j(x)e^{-L_j(x)}-1| > K\eta^{1/K} \}| \le  K\delta|B|.\]
\end{lemma}

We will use this lemma in order to obtain structure for $g$, as described in the following proposition.

\begin{proposition}\label{nearextg} Let $d\ge 1$ and let $q\ge 4$ be an even integer with conjugate exponent $p$. There exist positive constants $\tilde{K}$ $\delta_0>0$, depending only on $d$, with the following property. Suppose that $(1-\delta){\bf{B}}_{q,d}|\B|^{1/p}\le \|e^{ig}1_{\B}\|_q$ for $\delta\le\delta_0$ and $g$ real-valued. Then there exists an affine function $L_1:\R^d\to\R$ such that
\begin{align*} 
\int_{\B}|e^{iL_1(x)-ig(x)}-1|dx&\le \tilde{K}\delta^{1/(8\tilde{K})}.
\end{align*} 
\end{proposition}

\begin{proof} Let $q=2m$. We use the expression
\begin{align*} 
\|\vwidehat{e^{ig}1_{\B}}\|_{q}^{q}=\int_{\B^{q-1}} \cos(g(x_1)+\cdots +g(x_m)-g(y_2)-\cdots -g(y_m)-g(L(x,y)))1_{\B}(L(x,y)) dxdy 
\end{align*}
where $x=(x_1,\ldots,x_m)\in \R^{md}$, $y=(y_2,\ldots,y_m)\in\R^{(m-1)d}$, and $L(x,y)=x_1+\cdots +x_m-y_2-\cdots y_m$. Let $A(x,y)=g(x_1)+\cdots +g(x_m)-g(y_2)-\cdots -g(y_m)-g(L(x,y))$. Since $e^{ig}1_{\B}$ is a near-extremizer, we have
\begin{align*}
(1-\delta){\bf{B}}_{q,d}|\B|^{1/{p}}&\le \|\vwidehat{e^{ig}1_{\B}}\|_{q}^{q}= \|\widehat{1_{\B}}\|_{q}^{q}+\|\vwidehat{e^{ig}1_{\B}}\|_{q}^{q}-\|\widehat{1_{\B}}\|_{q}^{q}\\
&= {\bf{B}}_{q,d}-\int_{\B^{q-1}} |\cos(A(x,y))-1|1_{\B}(L(x,y)) dxdy,
\end{align*}
so $\int_{\B^{q-1}}|\cos(A(x,y))-1|1_{\B}(L(x,y))dxdy\le {\bf{B}}_{q,d}|\B|^{1/p}\delta$. We use this in the following:

\begin{align*}
\int_{\B^{q-1}} |e^{iA(x,y)}-1|1_{\B}(L(x,y)) dxdy&=   \int_{\B^{q-1}} (\cos(A(x,y))-1)^2+(\sin(A(x,y)))^2)^{1/2}1_{\B}(L(x,y)) dxdy  \\
    &=   \int_{\B^{q-1}} \sqrt{2} |\cos(A(x,y))-1|^{1/2}1_{\B}(L(x,y)) dxdy \\
    &\le \sqrt{2}|\B^{q-1}|^{1/2}\left(\int_{\B^{q-1}}  |\cos(A(x,y))-1|1_{\B}(L(x,y)) dxdy\right)^{1/2}\\
    &\le \sqrt{2}|\B|^{(q-1)/2}({\bf{B}}_{q,d}|\B|^{1/p}\delta)^{1/2}.
\end{align*}
Set 
\begin{align}\tilde{\delta}= \int_{\B^{q-1}} |e^{iA(x,y)}-1|1_{\B}(L(x,y)) dxdy\le C\delta^{1/2}. \label{goodolddays}\end{align}
Note that sometimes we will abuse the notation $L$: $L(x,y)$ where $x\in\R^{md}$, $y\in\R^{(m-1)d}$ and $L(x_1+x_2,x',y)$ where $x_1,x_2\in\R^d$, $x'\in\R^{(m-2)d}$, $y\in\R^{(m-1)d}$ mean the same thing. Define the function $\a$ by
\begin{align} 
\a(x_1+x_2,x',y):= g(x_3)+\cdots +g(x_m)-g(y_2)-\cdots -g(y_m)-g(L(x_1+x_2,x',y)) \label{defalpha}  
\end{align}
so that $A(x,y)= g(x_1)+g(x_2)+\a(x_1+x_2,x',y)$ for $L(x,y)\in \B$. Define the set $S_{\tilde{\delta}}\subset\R^{(q-3)d}$ to be
\[ S_{\tilde{\delta}}:=\left\{(x',y)\in \B^{q-3}:\int_{\B^2}|e^{i(g(x_1)+g(x_2)+\a(x_1+x_2,x',y))}-1|1_{\B}(L(x,y))dx_1dx_2>{\tilde{\delta}}^{1/2}\right\}. \]

Using Chebyshev's inequality in (\ref{goodolddays}), we know that $|S_{\tilde{\delta}}|\le {\tilde{\delta}}^{1/2}$. Fix an $(x',y)\in \B^{q-3}\setminus S_{\tilde{\delta}}$ such that $|(x',y)|<2\inf\{|(w',z)|: (w',z)\in \B^{q-3}\setminus S_{\tilde{\delta}}\}$. Since $|S_{\tilde{\delta}}|\le {\tilde{\delta}}^{1/2}$, there must be some positive intersection between $\B^{q-3}\setminus S_{\tilde{\delta}}$ and the ball in $\R^{(q-3)d}$ centered at the origin of radius $c^{-1/((q-3)d)}_{(q-3)d}(2{\tilde{\delta}})^{1/(2(q-3)d)}$, where $c_{(q-3)d}$ is the volume of the $(q-3)d$-dimensional unit ball. Thus 
\[ |(x',y)|\le 2c^{-1/((q-3)d}_{(q-3)d}(2{\tilde{\delta}})^{1/(2(q-3)d)}=: b(q,d) {\tilde{\delta}}^{1/(2(q-3)d)}.\]
For our fixed $(x',y)$, let $a$ denote 
\begin{align}
a:=x_3'+\cdots x_m'-y_2-\cdots-y_m.\label{defa}
\end{align} Note that 
\begin{align*} 
|a|&\le |x_3'|+\cdots+ |x_m'|+|y_2|+\cdots+|y_m| \\
    &\le (q-3)^{1/2}(|x_3'|^2+\cdots+ |x_m'|^2+|y_2|^2+\cdots+|y_m|^2)^{1/2}\\
    &=(q-3)^{1/2}|(x',y)|^{1/2}<(q-3)^{1/2}b(q,d) {\tilde{\delta}}^{1/(2(q-3)d)}.  
\end{align*}

Since $(x',y)\not\in S_{\tilde{\delta}}$, we apply Chebyshev's inequality again to get
\begin{align} 
|\{(x_1,x_2)\in\B^2:|e^{i(g(x_1)+g(x_2)+\a(x_1+x_2,x',y))}-1|1_{\B}(L(x,y))>{\tilde{\delta}}^{1/4}\} | < {\tilde{\delta}}^{1/4} \label{assump}. 
\end{align}

In order to satisfy the hypotheses of Lemma \ref{younglem}, we need to eliminate the indicator function $1_{\B}(L(x,y))$ from the set. We accomplish this by shrinking the size of the ball. Indeed let $r:= \frac{1}{2}(1-|a|)$. It is now clear why we chose $(x',y)$ with nearly minimal modulus: to make $r$ reasonably close to $1/2$. Then  
\begin{align*}  
&|\{(x_1,x_2)\in B(r)^2: |e^{i(g(x_1)+g(x_2)+\a(x_1+x_2,x',y))}-1|>{\tilde{\delta}}^{1/4} \}| \\
&\quad\le | \{(x_1,x_2)\in \B^2: |e^{i(g(x_1)+g(x_2)+\a(x_1+x_2,x',y))}-1|1_{\B}(L(x,y))>{\tilde{\delta}}^{1/4}\}|\le {\tilde{\delta}}^{1/4} , 
\end{align*} 
where $B(r)$ denotes the $d$-dimensional ball of radius $r$ centered at the origin. 
Thus the hypotheses of Lemma \ref{younglem} are satisfied. The lemma guarantees the existence of an affine function $L_0:\R^d\to\C$ such that 

\begin{align}   |\{x\in B(r): |e^{ig(x)}e^{-L_0(x)}-1|>K{\tilde{\delta}}^{1/(4K)}\}|\le K{\tilde{\delta}}^{1/4}r^{-d}  \label{L_0}\end{align}
where $K>0$ depends only on the dimension. Note that there exists a constant $c>0$ such that $|e^{ig(x)}e^{-L_0(x)}-1|\ge c|e^{ig(x)}e^{-i\text{Im}L_0(x)}-1|$ for all $x\in \B$. Then for $K'= K/c$, it follows from (\ref{L_0}) that 

\[ |\{x\in B(r): |e^{ig(x)-i\Im\, L_0(x)}-1|>K'{\tilde{\delta}}^{1/(4K)}\}|\le K{\tilde{\delta}}^{1/4}r^{-d}  . \label{ImL_0} \]

Since we know that $e^{ig}\approx e^{i\Im\, L_0}$ on the majority of $B(r)$, we can use the set whose measure is bounded in (\ref{assump}) to make conclusions about $\a$ on the set $B(r)+B(r)$ (which by the definition of $\a$ gives us information about $g$ on $B(2r)$). More precisely, we have that 

\begin{align}
|\{(x_1,x_2)\in& B(r)^2:|e^{i\Im\, L_0(x_1+x_2)+i\a(x_1+x_2,x',y)}-1|>4K'{\tilde{\delta}}^{1/(4K)}\} |\le \nonumber \\
&\quad |\{(x_1,x_2)\in B(r)^2:|e^{i\Im\, L_0(x_1+x_2)-ig(x_1)-ig(x_2)}-1|>2K'{\tilde{\delta}}^{1/(4K)}\} |  \nonumber \\
    &\quad+|\{(x_1,x_2)\in B(r)^2:|e^{i(g(x_1)+g(x_2)+\a(x_1+x_2,x',y))}-1|>2K'{\tilde{\delta}}^{1/(4K)}\} | \nonumber \\
&\le 2|\{(x_1,x_2)\in B(r)^2:|e^{i\Im\, L_0(x_1)-ig(x_1)}-1|>K'{\tilde{\delta}}^{1/(4K)}\} | +{\tilde{\delta}}^{1/4} \nonumber \\
&\le 2 c_dr^d K{\tilde{\delta}}^{1/4}r^{-d}  +{\tilde{\delta}}^{1/4} (2c_dK+1){\tilde{\delta}}^{1/4} \label{alphaL_0bound}
\end{align}
where $c_d$ is the volume of the $d$-dimensional unit ball. 

Recalling the definition of $\a$ from (\ref{defalpha}) and $a$ from (\ref{defa}) we write for $x_1,x_2\in B(r)$ 
\begin{align*}
e^{i\Im\, L_0(x_1+x_2)+i\a(x_1+x_2,x',y)}&=e^{i\Im\, L_0(x_1+x_2)+ig(x_3)+\cdots+ig(x_m)-ig(y_2)-\cdots-ig(y_m) -ig(L(x_1+x_2,x',y)) }\nonumber \\
    &=e^{i\Im\, L_0(x_1+x_2)+ig(x_3)+\cdots+ig(x_m)-ig(y_2)-\cdots-ig(y_m) -ig(x_1+x_2+a) }\nonumber.
\end{align*}
Combining this expression with the bound from line (\ref{alphaL_0bound}), we have actually shown that for $L_1:\R^d\to\R$ the affine function defined by $L_1(w)=\Im\, L_0(w)-\Im\, L_0(a)+ig(x_3)+\cdots+ig(x_m)-ig(y_2)-\cdots-ig(y_m)$,
\begin{align}
    |\{(x_1,x_2)\in B(r)^2: |e^{iL_1(x_1+x_2+a) -ig(x_1+x_2+a) }-1|>4K'{\tilde{\delta}}^{1/4}\}|&\le (2c_dK+1){\tilde{\delta}}^{1/4} .\label{prodbound}
\end{align}
Let $A$ denote the set whose measure is bounded above in (\ref{prodbound}). Let $E\subset \R^d$ denote the set 
\[ E:=\{x\in B(2r):|e^{L_1(w+a) -ig(w+a) }-1|>4K'{\tilde{\delta}}^{1/4}\}. \]

Writing line (\ref{prodbound}) in reverse order, we relate $A$ to $E$ as follows.
\begin{align*}
(2c_dK+1){\tilde{\delta}}^{1/4}&\ge \iint_{B(r)^2}1_{A}(x_1,x_2)dx_1dx_2\\
    &=  \iint_{B(r)^2}1_{E}(x_1+x_2)dx_1dx_2\\
    &= \langle 1_{B(r)}*1_{B(r)},1_E\rangle \\
    &\ge \left(\inf\{1_{B(r)}*1_{B(r)}(x):x\in B(1-\eta)\}\right)|E\cap B(1-\eta)|.   
\end{align*}
Note that there is a dimensional constant $a_d>0$ such that $1_{\B}(r)*1_{\B}(r)(x)>a_d\eta$ if $\eta<1$ and $x\in B(1-\eta)$. Thus if we pick $\eta={\tilde{\delta}}^{1/8}$,
\begin{align*}
|E\cap B(1-{\tilde{\delta}}^{1/8})|&\le a_d^{-1}(2c_dK+1){\tilde{\delta}}^{1/8}. 
\end{align*}
Putting everything together, we can now bound the $L^1$ norm of $|e^{iL_1-ig}-1|$:

\begin{align*} 
\int_{\B}|e^{iL_1(x)-ig(x)}-1|dx&=\int_{E\cap B(1-{\tilde{\delta}}^{1/8})}|e^{iL_1(x)-ig(x)}-1|dx+\int_{E\setminus B(1-{\tilde{\delta}}^{1/8})}|e^{L_1(x)-ig(x)}-1|dx\\
    &\quad +\int_{\B\setminus E} |e^{L_1(x)-ig(x)}-1|dx\\
    &\le \left[a_d^{-1}(2c_dK+1){\tilde{\delta}}^{1/8}+c_d(d+1){\tilde{\delta}}^{1/8}\right](2)+c_d4K'{\tilde{\delta}}^{1/(4K)}. 
\end{align*} 

\end{proof}

Finally, we consider maximizers of the form $e^{ig}1_{\B}$.

\begin{lemma}\label{extg} Let $d\ge 1$. Suppose that $q\ge 4$ is an even integer and that $\|\widehat{e^{ig}1_{\B}}\|_q^q={\bf{B}}_{q,d}$. Then there exists an affine function $L:\R^d\to\R$  such that $e^{ig}=e^{iL}$ on $\B$.

\end{lemma}

This lemma follows from a standard argument. Using the expression from (\ref{star}), the equality $\|\widehat{e^{ig}1_{\B}}\|_q={\bf{B}}_{q,d}=\|\widehat{1_{\B}}\|_q$ leads to a functional equation. By taking smooth approximations (which still satisfy the functional equation) and using derivatives, we find that $e^{ig}$ has the desired form.

\subsection{Proof strategy for Theorem \ref{mainthm}.\label{suff}}

Fix a dimension $d\ge 1$ and an even integer $\overline{q}\ge 4$. By Lemma \ref{sufficient}, to prove Theorem \ref{mainthm}, it is sufficient to find $\delta>0$ and $\rho>0$ so that if $|q-\overline{q}|<\rho$ and if
\begin{align}  \|\widehat{fe^{ig}1_E}\|_q^q \ge {\bf{B}}_{q,d}^q-\delta  \end{align}
where $f,g$ are real valued with $0\le f\le 1$ and $E\subset\R^d$ is a Lebesgue measurable subset with $|E|=1$, then 
\begin{align*} \|\widehat{fe^{ig}1_E}\|_q^q\le {\bf{B}}_{q,d}^q-c_{q,d}\left[ \|f-1\|_{L^1(E)}+\text{dist}_E(e^{ig},\mathfrak{L})^2+\text{dist}(E,\mathfrak{E})^2 \right] .    \end{align*}
By Corollary \ref{suffcor}, the hypothesis $\|\widehat{fe^{ig}1_E}\|_q^q\ge {\bf{B}}_{q,d}^q-\delta$ implies that 
\[ \|\widehat{fe^{ig}1_E}\|_{\overline{q}}^{\overline{q}}\ge {\bf{B}}_{\overline{q},d}^{\overline{q}}-o_{q-\overline{q}}(1)-\delta .\]

If $\delta,\rho$ are sufficiently small, then by the previous section, there exists an affine automorphism $T:\R^d\to\R^d$ so that $|T^{-1}(E)|=|\B|$ and $|T^{-1}(E)\Delta\B|\le 2\text{dist}(E,\mathfrak{E})\le C_{\overline{q},d}(o_{q-\overline{q}} (1)+ \delta^{1/2})$. Let $\overline{p}$ be the conjugate exponent to $\overline{q}$.

Define $f':\B\to [0,1]$ by $f\circ T$ on $T^{-1}(E)\cap\B$ and by $1$ on $\B\setminus T^{-1}(\B)$. By Lemma \ref{nearextf},
\begin{align*} 
\|\widehat{1_\B}\|_{\overline{q}}^{\overline{q}}-c_{\overline{q}}\|f'-1\|_{L^1(\B)}&\ge \|\widehat{f'1_{\B}}\|_{\overline{q}}^{\overline{q}}\\
&\ge |\B|^{\overline{q}/\overline{p}}{\bf{B}}_{\overline{q},d}^{\overline{q}}-o_{q-\overline{q}}(1)-o_{\delta}(1)\\
&=\|\widehat{1_{\B}}\|_{\overline{q}}^{\overline{q}}-o_{q-\overline{q}}(1)-o_{\delta}(1).
\end{align*} 
Thus $\|f\circ T-1\|_{L^1(T^{-1}(E))}\le \|f'-1\|_{L^1(\B)}+\|f\circ T-1\|_{L^1(T^{-1}(E)\setminus\B)}\le o_{q-\overline{q}}(1)+o_{\delta}(1)$.

Define $g':\B\to\R$ by $g\circ T$ on $T^{-1}(E)\cap\B$ and $0$ on $\B\setminus T^{-1}(E)$. Since $\|(f\circ Te^{ig\circ T}1_E\circ T)^{\widehat{\,\,}}\|_{\overline{q}}/|T^{-1}(E)|^{1/\overline{p}}= \|(fe^{ig}1_E)^{\widehat{\,\,}}\|_{\overline{q}}$,  

\begin{align*}
    \|\widehat{e^{ig'}1_{\B}}\|_{\overline{q}}&\ge\|(f\circ Te^{ig\circ T}1_E\circ T)^{\widehat{\,\,\,}}\|_{\overline{q}}-\|f\circ Te^{ig\circ T}1_E\circ T-e^{ig'}1_{\B}\|_{\overline{p}}\\
    &\ge |T^{-1}(E)|^{1/\overline{p}}{\bf{B}}_{\overline{q},d}-o_{q-\overline{q}}(1)-o_{\delta}(1)\\
    &\quad -\|f\circ T-1\|_{L^1(E)}^{1/\overline{p}}-|T^{-1}(E)\Delta\B|^{1/\overline{p}} \\
    &= |\B|^{1/\overline{p}}{\bf{B}}_{\overline{q},d}-o_{q-\overline{q}}(1)-o_{\delta}(1). 
\end{align*}
Then Proposition \ref{nearextg} applies, so there exists a real-valued affine function $L:\R^d\to\R^d$ so that 
\[ \|e^{ig'}-e^{iL}\|_{L^1(\B)}\le o_{q-\overline{q}}(1)+o_{\delta}(1). \]
If for a.e. $x\in T^{-1}(\B)$ we choose a representative of the equivalence class $[g\circ T(x)-L(x)]\in\R/(2\pi)$ with values in some range $[-M,M]$, then 
\[ \|g\circ T-L\|_{L^2(T^{-1}(E))}^2\le M^2|T^{-1}(E)\setminus\B|+M\|e^{ig'}-e^{iL}\|_{L^1(\B)}\le M^2(o_{q-\overline{q}}(1)+o_{\delta}(1)).  \]

We will prove in Propositions \ref{modulus}, \ref{support}, and \ref{freq} that 

\[ \|(f\circ Te^{i(g\circ T-L)}1_E\circ T)^{\widehat{\,\,}}\|_q^q\le \|\widehat{1_{\B}}\|_q^q-c_{q,d}\left[ \|f\circ T-1\|_{L^1(T^{-1}(E))}+\text{dist}_{T^{-1}(E)}(e^{ig\circ T},\mathfrak{L})^2+|T^{-1}(E)\Delta \B|^2 \right] . \]
Since $\|(f\circ Te^{i(g\circ T-L)}1_E\circ T)^{\widehat{\,\,}}\|_{q}/|\B|^{1/p}= \|(fe^{ig}1_E)^{\widehat{\,\,}}\|_{q}$, it follows that 
\begin{align*}  
\|\widehat{fe^{ig}1_E}\|_q^q&\le {\bf{B}}_{q,d}^q-|\B|^{-1/p}c_{q,d}\left[ \|f\circ T-1\|_{L^1(T^{-1}(E))}+\text{dist}_{T^{-1}(E)}(e^{ig\circ T},\mathfrak{L})^2+|T^{-1}(E)\Delta \B|^2 \right] \\
&\le  {\bf{B}}_{q,d}^q-|\B|^{1/q}c_{q,d}\left[ \|f-1\|_{L^1(E)}+\text{dist}_E(e^{ig},\mathfrak{L})^2+\text{dist}(E,\mathfrak{E})^2 \right] 
\end{align*} 
where we used that $|\B|=|T^{-1}(E)|=|\det T^{-1}||E|=|\det T^{-1}|$. 







\begin{section}      {\label{taylorexpansion}A Taylor expansion representation of $\|\widehat{fe^{ig}1_E}\|_q^q$.} 

Assuming that $|f|\le 1_E$ is close to $1_{\B}$ in the appropriate sense, we can find a good representation of $\|\widehat{f}\|_q^q$ using a Taylor expansion about $\|\widehat{1_{\B}}\|_q^q$. This is analogous to Lemma 3.4 in \cite{c2}. The functions in the following definition arise in the Taylor expansion. 
\begin{definition} For $d\ge 1$ and $q\in (3,\infty)$, we define the functions $K_q$ and $L_q$ on $\R^d$ by 
\begin{align}
    \widehat{K_q}=|\widehat{1_{\B}}|^{q-2}\widehat{1_{\B}} \label{Kq}\\
\widehat{L_q}=|\widehat{1_{\B}}|^{q-2} .\label{Lq}
\end{align}
\end{definition}
The basic properties of $K_q$ and $L_q$ are discussed in \textsection\ref{KqLq} below. For any function $f$ on $\R^d$, we let $\tilde{f}$ be the function $\tilde{f}(x)=f(-x)$. 

\begin{lemma}\label{Taylorgen}  Let $d\ge1$ and $q\in (3,\infty)$ with conjugate exponent $q'$. Let $E\subset{\R^d}$ and $|f|\le 1_E$. Set $h=f1_E-1_{\B}$. For sufficiently small $\|h\|_{q'}$, 
\begin{align*}
\|\widehat{f}\|_q^q    &= \|\widehat{1_{\B}}\|_q^q+q\langle K_q, \textup{Re}\, h\rangle -\frac{1}{4}q(q-2)\langle \textup{Im}\, h*\textup{Im}\, h,L_q\rangle+\frac{1}{4}q^2\langle \textup{Im}\, h,\textup{Im}\, h*L_q\rangle\\
    &\quad +\frac{1}{4}q(q-2)\langle \textup{Re}\, h*\textup{Re}\, h,L_q\rangle +\frac{1}{4}q^2\langle \textup{Re}\, h*\textup{Re}\,\tilde{h},L_q\rangle+O(\|h\|_{q'}^3) .
\end{align*}
\end{lemma}

If $q$ belongs to a compact subset of $(3,\infty)$, the constant implicit in the notation $O(\cdot)$ may be taken to be independent of $q$.

\begin{proof}
Using the Taylor expansion 
\[  |1+t|^q=1+q\Re\,t+\frac{1}{2}q(q-1)(\Re\,t)^2+\frac{1}{2}q(\Im\,t)^2+O(|t|^3+|t|^q) \]
valid for $q\in (3,\infty)$ and the fact that $\widehat{1_{\B}}$ is real-valued, we have
\begin{align}
    |\widehat{1_{\B}}+\widehat{h}|^q&=|\widehat{1_{\B}}|^q+q(\Re\,\widehat{h})\widehat{1_{\B}}|\widehat{1_{\B}}|^{q-2}\nonumber +\frac{1}{2}q(q-1)(\Re\,\widehat{h})^2|\widehat{1_{\B}}|^{q-2} \\
    &+\frac{1}{2}q(\Im\,\widehat{h})^2|\widehat{1_{\B}}|^{q-2}+O(|\widehat{h}|^3|\widehat{1_{\B}}|^{q-3})+O(|\widehat{h}|^q). \nonumber  
\end{align}
Next we integrate over $\R^d$ to obtain 
\begin{align*}
    \|\widehat{f}\|_q^q &= \|\widehat{1_{\B}}\|_q^q+q\langle K_q, \Re\, h\rangle +\frac{1}{2}q(q-1)\langle(\Re\,\widehat{h})^2,\widehat{L_q}\rangle\\
    &\quad+\frac{1}{2}q\langle(\Im\,\widehat{h})^2,\widehat{L_q}\rangle+O(\|h\|_{q'}^3+\|h\|_{q'}^q) \\
    &= \|\widehat{1_{\B}}\|_q^q+q\langle K_q, \Re\, h\rangle +\frac{1}{8}q(q-1)\langle(\widehat{h}+\overline{\widehat{h}})^2,\widehat{L_q}\rangle\\
    &\quad-\frac{1}{8}q\langle(\widehat{h}-\overline{\widehat{h}})^2,\widehat{L_q}\rangle+O(\|h\|_{q'}^3).
\end{align*} 
Using Plancherel's theorem, recalling that the $\tilde{\cdot}$ notation denotes the reflected function, using that $L_q$ satisfies $\tilde{L_q}=L_q$, and exploiting the equality $\langle f_1*f_2,f_3\rangle=\langle f_1,\tilde{f}_2*f_4\rangle$ for real-valued functions $f_i$, we further compute 
\begin{align*} 
\|\widehat{f}\|_q^q    &= \|\widehat{1_{\B}}\|_q^q+q\langle K_q, \Re\, h\rangle +\frac{1}{8}q(q-1)\langle(h+\overline{\tilde{h}})*(h+\overline{\tilde{h}}),L_q\rangle\\
    &\quad-\frac{1}{8}q\langle(h-\overline{\tilde{h}})*(h-\overline{\tilde{h}}),L_q\rangle+O(\|h\|_{q'}^3) \\
    &= \|\widehat{1_{\B}}\|_q^q+q\langle K_q, \Re\, h\rangle +\frac{1}{8}q(q-1)\langle(h*h+2h*\overline{\tilde{h}}+\overline{\tilde{h}}*\overline{\tilde{h}}),L_q\rangle \\
    &\quad -\frac{1}{8}q\langle(h*h-2h*\overline{\tilde{h}}+\overline{\tilde{h}}*\overline{\tilde{h}}),L_q\rangle+O(\|h\|_{q'}^3) \\
    &= \|\widehat{1_{\B}}\|_q^q+q\langle K_q, \Re\, h\rangle +\frac{1}{8}q(q-2)\langle(h*h+\overline{\tilde{h}}*\overline{\tilde{h}}),L_q\rangle \\
    &\quad +\frac{1}{4}q^2\langle h*\overline{\tilde{h}},L_q\rangle+O(\|h\|_{q'}^3) \\
    &= \|\widehat{1_{\B}}\|_q^q+q\langle K_q, \Re\, h\rangle +\frac{1}{8}q(q-2)\langle 2\Re\, h*\Re\, h-2\Im\, h*\Im\, h,L_q\rangle \\
    &\quad +\frac{1}{4}q^2\langle \Re\, h*\Re\,\overline{\tilde{h}}-\Im\, h*\Im\,\overline{\tilde{h}},L_q\rangle+O(\|h\|_{q'}^3) \\
    &= \|\widehat{1_{\B}}\|_q^q+q\langle K_q, \Re\, h\rangle -\frac{1}{4}q(q-2)\langle \Im\, h*\Im\, h,L_q\rangle+\frac{1}{4}q^2\langle \Im\, h,\Im\, h*L_q\rangle\\
    &\quad +\frac{1}{4}q(q-2)\langle \Re\, h*\Re\, h,L_q\rangle +\frac{1}{4}q^2\langle \Re\, h*\Re\,\tilde{h},L_q\rangle+O(\|h\|_{q'}^3) .
\end{align*}

\end{proof}

\begin{subsection}{\label{KqLq}The functions $K_q$ and $L_q$.}
Let $K_q$ and $L_q$ be the functions defined in (\ref{Kq}) and (\ref{Lq}).  We state the facts proved in in \cite{c2} about $K_q$ and $L_q$ that we will need for our analysis of the Taylor expansion. See \textsection 3.3 and \textsection 3.4 from \cite{c2} for detailed explanations.

As is well known, $\widehat{1_{\B}}$ is a radially symmetric real-valued real analytic function which satisfies 
\begin{align} |\widehat{1_{\B}}(\xi)|+|\nabla \widehat{1_{\B}}(\xi)|\le C_d(1+|\xi|)^{-(d+1)/2} .\label{cpct} \end{align}

The following are consequences of (\ref{cpct}):

\begin{lemma} Let $d\ge 1$ and $q\in(3,\infty)$. The functions $K_q,L_q$ are real-valued, radially symmetric, bounded and H\"{o}lder continuous of some positive order. Moreover, $K_q(x)\to 0$ as $|x|\to\infty$ and likewise for $L_q(x)$. The function $K_q$ is continuously differentiable, and $x\cdot \nabla K_q$ is likewise real-valued, radially symmetric, and H\"{o}lder continuous of some positive order. These conclusions hold uniformly for $q$ in any compact subset of $(3,\infty)$. 

\end{lemma}

\begin{lemma}\label{Kqlem2}For each  $d\ge 1$, $K_q$, $L_q$, and $x\cdot \nabla_xK_q$ depend continuously on $q\in(3,\infty)$. This holds in the sense that for each compact subset $\Lambda\subset(3,\infty)$, the mappings $q\mapsto K_q$ and $q\mapsto L_q$ are continuous from $\Lambda$ to the space of continuous functions on $\R^d$ that tend to zero at infinity. Moreover, there exists $\rho>0$ such that this mapping from $\Lambda$ to the space of bounded H\"{o}lder continuous functions of order $\rho$ on any bounded subset of $\R^d$ is continuous. The two conclusions also hold for $q\mapsto x\cdot \nabla_x K_q$. 
\end{lemma}

The following lemma is an immediate consequence of the boundedness of $L_q$. 
\begin{lemma}Let $d\ge 1$ and $q\in(3,\infty)$. Let $f1_E$ be a Lebesgue-measurable function with $|E|\in\R^+$ and $|f|\le 1$. Let $h_1,h2\in L^1(\R^d)$.  Then 
\[ \langle h_1*L_q,h_2\rangle = O(\|h_1\|_1\|h_2\|_1). \]

\end{lemma}

\begin{lemma}\label{Kqprop} For each $d\ge 1$ and each even integer $m\ge 4$ there exists $\eta=\eta(d,m)>0$ such that whenever $|q-m|<\eta$, there exists $c>0$ such that whenever $|y|\le 1\le |x|\le 2$, 
\begin{align}
K_q(y)\ge K_q(x)+c||x|-|y|| \label{a1}.
\end{align}
Also, $\inf_{x\in 2\B}K_q(x)>0$ and 
\begin{align}
\min_{|x|\le 1-\delta}K_q(x)>\max_{|x|\ge  1+\delta}K_q(x)\quad\text{for all }\delta>0. \label{a2}
\end{align}

\end{lemma}

\begin{proof} From the proof of Lemma 3.10 in \cite{c2}, if $m$ is an even integer greater than 3 and $|y|=1$, then the map $t\mapsto K_m(ty)$ has strictly negative derivative for all $t\in(0,m-1)$. The inequalities (\ref{a1}) and  (\ref{a2}) are a direct result. Since $\inf_{x\in 2\B}K_q(x)=\inf_{x\in2\B}1_{\B}*\cdots*1_{\B}(x)$ is obviously positive when $q$ is even, the same holds for near $q$ by the continuity of $q\mapsto K_q$. 

\end{proof}

\begin{remark}We will use the fact that $\inf_{x\in 2\B}K_q(x)$ is positive for $q$ near even integers extensively throughout the paper.
\end{remark}

\end{subsection}

\begin{subsection}{A more detailed Taylor expansion in terms of the support, frequency, and modulus.}

In order to better understand the effects of specific variations of $1_{\B}$, we consider $fe^{ig}1_E$ where $0\le f\le 1$, $g$ is real valued, and $|E|\in\R^+$.

\begin{lemma} \label{Taylorfg}  Let $d\ge1$ and $q\in (3,\infty)$. Let $E\subset{\R^d}$. Suppose that $|E\Delta\B|$, $\|g\|_{L^2(E)}$, and $\|f-1\|_{L^1(E)}$ are sufficiently small. Then

\begin{align*} 
\|\widehat{fe^{ig}1_E}\|_q^q&=\|\widehat{1_E}\|_q^q+q\langle K_q, f\cos g1_{E\setminus \B}-1_{E\setminus \B}\rangle-\|\widehat{1_{\B}}\|_q^q+\|\widehat{f'e^{ig'}1_{\B}}\|_q^q \\
    &\quad+ O((\|g\|_2^2+\|f-1\|_{L^1(E)})|E\Delta\B|^{1/2})+O(\|g\|_{L^2(E)}^3+\|f-1\|^{2}_{L^1(E)}+|E\Delta\B|^{3/q'}) 
\end{align*} 
where $g'=g$ on $E\cap \B$, $g'=0$ on $\R^d\setminus(E\cap \B)$ and  $f'=f$ on $E\cap \B$, $f'=1$ on $\B\setminus E$, $f'=0$ on $\R^d\setminus\B$.

\end{lemma}

\begin{proof} As in Lemma \ref{Taylorgen}, set $h=fe^{ig}1_E-1_{\B}$. Using the expression for $\|\widehat{fe^{ig}1_E}\|_q^q$ from Lemma \ref{Taylorgen}, we replace some of the terms with $h$ by $h-1_E+1_E$ and expand.

\begin{align}
\|\widehat{fe^{ig}1_E}\|_q^q  &= \|\widehat{1_{\B}}\|_q^q+q\langle K_q, f\cos g1_E-1_E+1_E-1_{\B}\rangle \nonumber\\
    &\quad -\frac{1}{4}q(q-2)\langle f\sin g1_E*f\sin g1_E,L_q\rangle +\frac{1}{4}q^2\langle f\sin g 1_E,f\sin g1_E*L_q\rangle\nonumber\\
    &\quad +\frac{1}{4}q(q-2)\langle (f\cos g1_E-1_E+1_E-1_{\B})*(f\cos g1_E-1_E+1_E-1_{\B}),L_q\rangle \nonumber\\
    &\quad +\frac{1}{4}q^2\langle (f\cos g1_E-1_E+1_E-1_{\B})*(\tilde{f}\cos \tilde{g}1_{\tilde{E}}-1_{\tilde{E}}+1_{\tilde{E}}-1_{\B}),L_q\rangle\nonumber\\
    &\quad+O(\|h\|_{q'}^3) \nonumber\\
&= \|\widehat{1_{\B}}\|_q^q+q\langle K_q, f\cos g1_E-1_E\rangle+q\langle K_q, 1_E-1_{\B}\rangle \nonumber\\
    &\quad -\frac{1}{4}q(q-2)\langle f\sin g1_E*f\sin g1_E,L_q\rangle +\frac{1}{4}q^2\langle f\sin g 1_E,f\sin g1_E*L_q\rangle\nonumber\\
    &\quad +\frac{1}{4}q(q-2)\langle (1_E-1_{\B})*(1_E-1_{\B}),L_q\rangle\nonumber \\
    &\quad +\frac{1}{4}q^2\langle (1_E-1_{\B})*(1_{\tilde{E}}-1_{\B}),L_q\rangle\nonumber\\
    &\quad+O(\|f\cos g1_E-1_E\|_1\|f\cos g1_E-1_{\B}\|_1+\|f\cos g1_E-1_E\|_1^2)+O(\|h\|_{q'}^3) \nonumber\\
&= \|\widehat{1_E}\|_q^q+q\langle K_q, f\cos g1_E-1_E\rangle\label{lastline} \\
    &\quad -\frac{1}{4}q(q-2)\langle f\sin g1_E*f\sin g1_E,L_q\rangle +\frac{1}{4}q^2\langle f\sin g 1_E,f\sin g1_E*L_q\rangle\nonumber\\
    &\quad+O(\|f\cos g1_E-1_E\|_1\|f\cos g1_E-1_{\B}\|_1)\nonumber\\
    &\quad+O(\|fe^{ig}1_E-1_E\|_{q'}^3+\|1_E-1_{\B}\|_{q'}^3) . \nonumber
\end{align}
where we used that Lemma \ref{Taylorgen} gives
\begin{align*}
\|\widehat{1_E}\|_q^q&= \|\widehat{1_{\B}}\|_q^q+q\langle K_q, 1_E-1_{\B}\rangle +\frac{1}{4}q(q-2)\langle (1_E-1_{\B})*(1_E-1_{\B}),L_q\rangle\nonumber \\
    & +\frac{1}{4}q^2\langle (1_E-1_{\B}),(1_E-1_{\B})*L_q\rangle+O(|E\Delta \B|^{3/q'}) .\nonumber 
\end{align*}
Lemma \ref{Taylorgen} also gives the following Taylor expansion for $\|\widehat{f'e^{ig'}1_{E}}\|_q^q$.

\begin{align*} 
\|\widehat{f'e^{ig'}1_\B}\|_q^q    &= \|\widehat{1_{\B}}\|_q^q+q\langle K_q, (f'\cos g'-1)1_{\B}\rangle \\
&-\frac{1}{4}q(q-2)\langle f'\sin g'1_{\B}*f'\sin g'1_{\B},L_q\rangle +\frac{1}{4}q^2\langle f'\sin g'1_{\B},f'\sin g'1_{\B}*L_q\rangle  \nonumber\\
&+\frac{1}{4}q(q-2)\langle (f'\cos g'-1)1_{\B}*(f'\cos g'-1)1_{\B},L_q\rangle\nonumber \\
    & +\frac{1}{4}q^2\langle (f'\cos g'-1)1_{\B},(f'\cos g'-1)1_{\B}*L_q\rangle+O(\|f'e^{ig'}-1\|_{L^{q'}(\B)}^3) \nonumber \\
&= \|\widehat{1_{\B}}\|_q^q+q\langle K_q, (f'\cos g'-1)1_{\B}\rangle \nonumber\\
&-\frac{1}{4}q(q-2)\langle f'\sin g'1_{\B}*f'\sin g'1_{\B},L_q\rangle +\frac{1}{4}q^2\langle f'\sin g'1_{\B},f'\sin g'1_{\B}*L_q\rangle  \nonumber\\
&+O(\|f'\cos g'-1\|_{L^1(\B)}^2)+O(\|f'e^{ig'}-1\|_{L^{q'}(\B)}^3) .\nonumber \\
\end{align*}
Using that $\|(f'e^{ig'}-1)1_{\B}\|_{q'}=\|(fe^{ig}-1)1_{E\cap \B}\|_{q'}\le \|(fe^{ig}-1)1_E\|_{q'}$ and that $\|f'\cos g'-1\|_{L^1(\B)}^2=\|f\cos g-1\|_{L^1(E\cap\B)}^2\le \|f\cos g1_E-1_E\|_1\|f\cos g1_E-1_{\B}\|_1 $, we will extract the expression for $\|\widehat{f'e^{ig'}1_{E}}\|_q^q$ above from some of the terms from  (\ref{lastline}).

\begin{align*}
q\langle K_q,& f\cos g1_E-1_E\rangle  -\frac{1}{4}q(q-2)\langle f\sin g1_E*f\sin g1_E,L_q\rangle +\frac{1}{4}q^2\langle f\sin g 1_E,f\sin g1_E*L_q\rangle\\
    &= q\langle K_q, f\cos g1_E-1_E-f'\cos g'1_{\B}+1_{\B}\rangle + q\langle K_q, f'\cos g'1_{\B}-1_{\B}\rangle   \\
    &\quad-\frac{1}{4}q(q-2)\langle (f\sin g1_E-f'\sin g'1_{\B}+f'\sin g'1_{\B})*(f\sin g1_E-f'\sin g'1_{\B}+f'\sin g'1_{\B}),L_q\rangle \\
    &\quad +\frac{1}{4}q^2\langle (f\sin g 1_E-f'\sin g'1_{\B}+f'\sin g'1_{\B}),(f\sin g1_E-f'\sin g'1_{\B}+f'\sin g'1_{\B})*L_q\rangle\\
&= q\langle K_q, f\cos g1_{E\setminus \B}-1_{E\setminus \B}-1_{B\setminus E}+1_{B\setminus E}\rangle + q\langle K_q, f\cos g1_{\B}-1_{\B}\rangle   \\
    &\quad-\frac{1}{4}q(q-2)\langle f'\sin g'1_{\B}*f'\sin g'1_{\B},L_q\rangle \\
    &\quad +\frac{1}{4}q^2\langle f'\sin g'1_{\B},f'\sin g'1_{\B}*L_q\rangle\\
    &\quad +O(\|f\sin g1_E-f'\sin g'1_{\B}\|_1\|f\sin g1_E\|_1+\|f\sin g1_E-f'\sin g'1_{\B}\|_1^2) \\
&=  q\langle K_q, f\cos g1_{E\setminus \B}-1_{E\setminus \B}\rangle + q\langle K_q, f'\cos g'1_{\B}-1_{\B}\rangle   \\
    &\quad-\frac{1}{4}q(q-2)\langle f'\sin g'1_{\B}*f'\sin g'1_{\B},L_q\rangle +\frac{1}{4}q^2\langle f'\sin g'1_{\B},f'\sin g'1_{\B}*L_q\rangle\\    
    &\quad +O(\|f\sin g1_{E\setminus \B}\|_1\|f\sin g1_E\|_1+\|f\sin g1_{E\setminus \B}\|_1^2) \\
&= q\langle K_q, f\cos g1_{E\setminus \B}-1_{E\setminus \B}\rangle + q\langle K_q, f'\cos g'1_{\B}-1_{\B}\rangle   \\
    &\quad-\frac{1}{4}q(q-2)\langle f'\sin g'1_{\B}*f'\sin g'1_{\B},L_q\rangle +\frac{1}{4}q^2\langle f'\sin g'1_{\B},f'\sin g'1_{\B}*L_q\rangle\\    
    &\quad +O(\|g\|_2^2|E\Delta\B|^{1/2})    \\
&=q\langle K_q, f\cos g1_{E\setminus \B}-1_{E\setminus \B}\rangle-\|\widehat{1_{\B}}\|_q^q+\|\widehat{f'e^{ig'}1_{\B}}\|_q^q\\
    &\quad+O(\|g\|_2^2|E\Delta\B|^{1/2})+O(\|f\cos g1_E-1_E\|_1\|f\cos g 1_E-1_{\B}\|_1)+O(\|(fe^{ig}-1)1_E\|_{q'}^3)  . 
\end{align*}

Using this simplified expression in \ref{lastline} gives
\begin{align*}
\|\widehat{fe^{ig}1_E}\|_q^q  &= \|\widehat{1_E}\|_q^q+q\langle K_q, f\cos g1_{E\setminus \B}-1_{E\setminus \B}\rangle-\|\widehat{1_{\B}}\|_q^q+\|\widehat{f'e^{ig'}1_{\B}}\|_q^q\\
    &\quad+O(\|g\|_2^2|E\Delta\B|^{1/2})+O(\|f\cos g1_E-1_E\|_1\|f\cos g1_E-1_{\B}\|_1)\\
    &\quad+O(\|fe^{ig}1_E-1_E\|_{q'}^3+\|1_E-1_{\B}\|_{q'}^3) ,
\end{align*}
so it remains to understand the big-O terms. 

Note that  
\[ \|fe^{ig}1_E-1_E\|_{q'}\le \|f-1\|^{1/q'}_{L^1(E)}+\|g\|_{L^2(E)}(2|E|)^{(2-q')/(2q')}. \] 
We also have that $\|f\cos g1_E-1_E\|_1\le\|g\|_{L^2(E)}^2+\|f-1\|_{L^1(E)}$ and
\[\|f\cos g1_E-1_{\B}\|_1\le \|g\|_{L^2(E)}^2+\|f-1\|_{L^1(E)}+|E\Delta\B|. \] 
Thus, noting that $3/q'>2$, we can simplify the big-O terms to 
\[ O((\|g\|_2^2+\|f-1\|_{L^1(E)})|E\Delta\B|^{1/2})+O(\|g\|_{L^2(E)}^3+\|f-1\|^{2}_{L^1(E)}+|E\Delta\B|^{3/q'}). \]

\end{proof}

\end{subsection}
\end{section}

\section{Mostly modulus variation: $\|f-1\|_{L^1(E)}^{1/2}\ge \max( MN|E\Delta\B|,N\|g\|_{L^2(E)})$}

Let $K_q,L_q$ be the functions defined in (\ref{Kq}) and (\ref{Lq}) in the following discussion. We use our most basic Taylor expansion from Lemma \ref{Taylorgen} to understand this case.

\begin{proposition}\label{modulus}
Let $d\ge 1$ and let ${\overline{q}}\ge 4$ be an even integer. There exist $M,N\in\R^+$, $\delta_0>0$, and $\rho>0$ all depending on $\overline{q}$ and $d$ such that the following holds. Let $q\in(3,\infty)$, $E\subset\R^d$ be a Lebesgue measurable set with $|E|\le |\B|$, $0\le f\le 1$, and $g$ be real valued. Suppose that  $\|f-1\|_{L^1(\B)}\le\delta_0$, $\|g\|_{L^2(E)}\le \delta_0$, $|E\Delta\B| \le\delta_0$, and $|q-{\overline{q}}|\le\rho $. If
\[\|f-1\|_{L^1(E)}^{1/2}\ge \max(MN|E\Delta\B|,N\|g\|_{L^2(E)}), \]
then 
\begin{align*}
\|\widehat{fe^{ig}1_{E}}\|_q^q&\le \|\widehat{1_{\B}}\|_q^q-c_{q,d}\|f-1\|_{L^1(E)}
\end{align*} 
for a constant $c_{q,d}>0$ depending only on the exponent $q$ and on the dimension. 

\end{proposition}

\begin{proof} We begin with the Taylor expansion for $\|\widehat{fe^{ig}1_E}\|_q^q$ from Lemma \ref{Taylorgen}.
\begin{align}
\|\widehat{fe^{ig}1_E}\|_q^q  &=\|\widehat{1_{\B}}\|_q^q+q\langle K_q, f\cos g1_E-1_{\B}\rangle \nonumber\\
    &\quad -\frac{1}{4}q(q-2)\langle f\sin g1_E*f\sin g1_E,L_q\rangle +\frac{1}{4}q^2\langle f\sin g 1_E,f\sin g1_E*L_q\rangle  \nonumber\\
    &\quad +\frac{1}{4}q(q-2)\langle (f\cos g1_E-1_{\B})*(f\cos g1_E-1_{\B}),L_q\rangle \nonumber\\
    &\quad +\frac{1}{4}q^2\langle (f\cos g1_E-1_{\B})*(\tilde{f}\cos \tilde{g}1_{\tilde{E}}-1_{\B}),L_q\rangle\nonumber\\
    &\quad+O(\|fe^{ig}1_E-1_{\B}\|_{q'}^3) \nonumber\\
&= \|\widehat{1_{\B}}\|_q^q+q\langle K_q, f\cos g1_E-1_{\B}\rangle \label{fra3}+O(\|\sin g\|_{L^1(E)}^2+\|f\cos g1_E-1_{\B}\|_1^2+\|fe^{ig}1_E-1_{\B}\|_{q'}^3)  . 
\end{align}

Analyze the inner product term in (\ref{fra3}).
  Let $G_{+}=\{x\in E:\cos g(x)\ge 0\}$, $G_{-}=\{x\in E:\cos g(x)<0\}$, $K_+=\{x\in\R^d:K_q(x)\ge 0\}$, and $K_{-}=\{x\in\R^d:K_q(x)<0\}$. The term $\langle K_q,f\cos g1_E-1_{\B}\rangle$ may be handled as follows:
    
\begin{align*}
    \langle K_q,f\cos g1_E-1_{\B}\rangle&= \langle K_q,f\cos g(1_{E\cap G_{+}\cap K_{+}}+1_{E\cap G_{+}\cap K_{-}}+1_{E\cap G_{-}\cap K_{+}}+1_{E\cap G_{-}\cap K_{-}})-1_{\B}\rangle \\
    &\le \langle K_q,f\cos g1_{E\cap G_{+}\cap K_{+}}-1_{\B}\rangle +\langle K_q,f\cos g 1_{E\cap G_{-}\cap K_{-}}\rangle \\
    &\le \langle K_q,f1_{E\cap G_{+}\cap K_{+}}-1_{\B}\rangle+C|G_{-}|
\end{align*}
where $C=C(q,d)>0$ is a constant. Note that on $G_{-}$, we must have $|g|\ge \pi/2$, so $|G_{-}|\le \frac{4}{\pi^2}\|g\|_{L^2(E)}^2\le \frac{4}{\pi^2N^2}\|f-1\|_{L^1(E)}$. Now consider $\langle K_q,f1_{E\cap G_{+}\cap K_{+}}-1_{\B}\rangle$: 

\begin{align*}
    \langle K_q,f&1_{(E\cap G_{+}\cap K_{+})\cap \B}-1_{\B\cap E}\rangle+\langle K_q,f1_{(E\cap G_{+}\cap K_{+})\setminus  \B}-1_{\B\setminus E}\rangle\\
    &\qquad\qquad  \le \langle K_q,f1_{E\cap \B}-1_{\B\cap E}\rangle+\langle K_q,f1_{(E\cap K_{+})\setminus  \B}-1_{\B\setminus E}\rangle \\
    &\le -\inf_{\B}K_q\cdot \|f-1\|_{L^1(E\cap \B)}+\sup_{|x|>1}K_q(x)\cdot \|f1_{E\setminus\B}\|_1-\inf_{\B}K_q\cdot |\B\setminus E| .
\end{align*} 
By Lemma \ref{Kqprop}, we know that $\inf_{\B}K_q=\sup_{|x|>1}K_q(x)=\left.K_q\right|_{\partial \B}$. We assumed that $|E|\le |\B|$ (so $-|\B\setminus E|\le -|E\setminus\B|$) in the hypotheses. Using these two observations, we further simplify and bound the above.  
\begin{align*} 
\sup_{|x|>1}K_q(x)\cdot \|f1_{E\setminus\B}\|_1-\inf_{\B}K_q\cdot |\B\setminus E|&\le     -\left.K_q\right|_{\partial\B}\left(|E\setminus \B|- \|f1_{E\setminus\B}\|_1\right) \\
&= -\left.K_q\right|_{\partial\B}\|f-1\|_{L^1(E\setminus \B)} .
\end{align*}
In summary, we have shown that 
\[ \langle K_q,f\cos g1_E-1_{\B}\rangle \le -\inf_{\B}K_q\cdot \|f-1\|_{L^1(E)}+C\frac{1}{N^2}\|f-1\|_{L^1(E)}.\]

Now analyze the error term in (\ref{fra3}). The error $O(\|f\cos g1_E-1_{\B}\|_1^2)$ can be replaced by $ O(\|f-1\|_{L^1(E)}^2+|E\Delta\B|^2)$ because
\begin{align*}  
\|f\cos g1_E-1_{\B}\|_1&\le \|f\cos g 1_E-\cos g1_E\|_1+\|\cos g1_E-1_E\|_1+|E\Delta \B|\\
    &\le \|f-1\|_{L^1(E)}+ \|g\|^2_{L^2(E)}+|E\Delta\B|, 
\end{align*} 
and similarly,  $O(\|fe^{ig}1_E-1_{\B}\|_{q'}^3)$ can be replaced by $O(\|f-1\|_{L^1(E)}^{3/(2q')})$ since
\begin{align*}
\|fe^{ig}1_E-1_{\B}\|_{q'}&\le \|fe^{ig}1_E-e^{ig}1_E\|_{q'}+\|e^{ig}1_E-1_E\|_{q'}+\|1_E-1_{\B}\|_{q'} \\
    &\le \|f-1\|_{L^1(E)} + \|g\|_{L^2(E)}|E|^{(q'-2)/(q')^2}+|E\Delta\B|^{1/q'}.
\end{align*}
Also note that since $|\sin x|\le| x|$, $\|\sin g\|_1^2\le\|g\|_1^2\le |\B|\|g\|_{L^2(E)}^2\le \frac{|\B|}{N^2}\|f-1\|_{L^1(E)}$.

Returning to (\ref{fra3}), we now have the following bound.

\begin{align*}
\|\widehat{fe^{ig}1_E}\|_q^q&= \|\widehat{1_{\B}}\|_q^q+q\langle K_q, f\cos g1_E-1_{\B}\rangle+O(\|\sin g\|_{L^1(E)}^2+\|f\cos g1_E-1_{\B}\|_1^2+\|fe^{ig}1_E-1_{\B}\|_{q'}^3) \\
&\le \|\widehat{1_{\B}}\|_q^q-q\inf_{ \B}K_q\cdot\|f-1\|_{L^1(E)} +C_1\left(\frac{1}{N^2}+\frac{1}{(MN)^2}\right)\|f-1\|_{L^1(E)}\\
    &\quad+O(\|f-1\|_{L^1(E)}^{3/(2q')})
\end{align*}
Thus for $M,N$ large enough and $\delta_0$ small enough, we have the desired conclusion. 

\end{proof}

The exponent $1$ and the $L^1(E)$ norm of $\|f-1\|_{L^1(E)}$ in the previous proposition are optimal in the following sense.

\begin{lemma}\label{optf} Under the hypotheses of Proposition \ref{modulus}, suppose for $n>0$ and $p\ge 1$ that 

\begin{align*} 
\|\widehat{f1_{\B}}\|_q^q &\le    \|\widehat{1_{\B}}\|_q^q-c_{q,d}\|f-1\|_{L^p(\B)}^n.
\end{align*}
Then $n\ge p$.

\end{lemma}

\begin{proof} Use Lemma \ref{Taylorgen} with $h=f1_{\B}-1_{\B}$:

\begin{align*}
\|\widehat{f1_{\B}}\|_q^q = \|\widehat{1_{\B}}\|_q^q+q\langle K_q, (f-1)1_{\B}\rangle +\frac{1}{4}q(q-2)\langle h*h,L_q\rangle  +\frac{1}{4}q^2\langle h*\tilde{h},L_q\rangle+O(\|h\|_{q'}^3+\|h\|_{q'}^q) .
\end{align*}
This combined with the hypothesis leads to 
\begin{align*}
c_{\overline{q},d}\|f-1\|_{L^p(\B)}^n& \le q\langle K_q, |f-1|1_{\B}\rangle+O(\|f-1\|_{L^1(\B)}^2+\|f-1\|_{L^1(\B)}^{q/q'}) .
\end{align*}
For all $\|f-1\|_{L^1(\B)}$ sufficiently small, this implies
\[ c_{\overline{q},d}\|f-1\|_{L^p(\B)}^n\le c\|f-1\|_{L^1(\B)}.\]
If we take $|f-1|\ge 1/2$, then 

\[2^{-(p-1)n/p}c_{q,d}\|f-1\|_{L^1(\B)}^{n/p}\le c_{q,d}\|f-1\|_{L^p(\B)}^n\le c\|f-1\|_{L^1(\B)}. \]
If we take $f=(1-\eta)$ and let $\eta\to 0$, then we must have $n/p\ge 1$. 

\end{proof}

Note that for $n\ge p$,
\[\|f-1\|_{L^p(E)}^n \le \|f-1\|_{L^1(E)}^{n/p} \le \|f-1\|_{L^1(E)} , \]
so Proposition \ref{modulus} is stronger than if $\|f-1\|_{L^1(E)}$ were replaced by $\|f-1\|_{L^p(E)}^n$.

\section{\label{genq}The special case $E=\B$ for $q$ near an even integer.}

Let $K_q,L_q$ be the functions defined in (\ref{Kq}) and (\ref{Lq}). In order to treat the remaining cases of mostly support variation (in \textsection\ref{supportsec}) and mostly frequency variation (in \textsection{\ref{freq2}}) of our near-extremizer $fe^{ig}1_E$ with $0\le f\le 1$, $g$ real-valued, begin with Lemma \ref{Taylorfg}. Let $g'=g$ on $E\cap \B$ and $g'=0$ on $B\setminus E$ and let $f'=f$ on $E\cap \B$ and $f'=1$ on $B\setminus E$. Recall the statement from Lemma \ref{Taylorfg} :

\begin{align*} 
\|\widehat{fe^{ig}1_E}\|_q^q&=\|\widehat{1_E}\|_q^q+q\langle K_q, f\cos g1_{E\setminus \B}-1_{E\setminus \B}\rangle-\|\widehat{1_{\B}}\|_q^q+\|\widehat{f'e^{ig'}1_{\B}}\|_q^q \\
    &\quad+ O((\|g\|_2^2+\|f-1\|_{L^1(E)})|E\Delta\B|^{1/2})+O(\|g\|_{L^2(E)}^3+\|f-1\|^{2}_{L^1(E)}+|E\Delta\B|^{3/q'}) . 
\end{align*}
In this section, we work to understand the term $\|\widehat{f'e^{ig'}1_{\B}}\|_q^q$ above.

\begin{subsection}{A new Taylor expansion for $\|\widehat{fe^{ig}1_{\B}}\|_q^q$ when $\|f-1\|_1,\|g\|_2\ll1$. }

The structural information we have obtained from Lemma \ref{nearextf} and Proposition \ref{nearextg} guarantees that if $fe^{ig}1_{\B}$ is a near-extremizer, then, after possibly replacing $g$ by an affine translation of $g$,  $fe^{ig}1_{\B}$ is reasonably close to $1_{\B}$. Thus a Taylor expansion of $\|\widehat{fe^{ig}1_{\B}}\|_q^q$ about $\|\widehat{1_{\B}}\|_q^q$ will have an error that we can control. Furthermore, since we know that $|e^{ig}-1|$ is small on the majority of $\B$, we can expand the $\sin(g)$ and $\cos(g)$ that appear in Lemma \ref{Taylorgen} using Taylor series approximations. We split up the set $\B$ into a subset where the frequency $|g|$ is small and the remainder set where the frequency is not small in order to use polynomials to approximate the trigonometric terms in the next lemma. Define 
\[ \B^\epsilon_g:= \{x\in \B: |g(x)|>\epsilon\} \]
for $0<\epsilon<\pi/2$ and $A_g:=\{x\in\B:\cos g(x)\ge 0\}$. Note that in the following lemma, we do not require a specific equivalence representative of $g(x)\in\R/(2\pi)$ for $x\in\B$.

\begin{lemma} \label{Taylorfreq}Let $d\ge 1$ and let $q\ge 4$ be an even integer. Let $f,g$ be real valued functions on $\B$ with $0\le f\le 1$. There exists $\delta_0>0$, depending on $q$ and on $d$, such that if $\|f-1\|_{L^1(\B)}\le\delta_0$ and $\|g\|_{L^2(\B)}\le \delta_0$, then 

\begin{align*}
\|\widehat{fe^{ig}1_{\B}}\|_q^q  &\le  \|\widehat{1_{\B}}\|_q^q-q\inf_{\B}K_q\cdot \left(\|\cos g-1\|_{L^1(A_g\cap \B_g^\epsilon)}+|\B_g^\epsilon\setminus A_g|\right) \\
    &-\frac{q}{2}\langle K_q, g^21_{\B\setminus \B^\epsilon_g}\rangle -\frac{1}{4}q(q-2)\langle g1_{\B\setminus \B^\epsilon_g}*g1_{\B\setminus \B^\epsilon_g},L_q\rangle+\frac{1}{4}q^2\langle  g1_{\B\setminus \B^\epsilon_g},g1_{\B\setminus \B^\epsilon_g}*L_q\rangle\\
    & +\epsilon^2O(\|g\|_{L^2(\B\setminus \B^\epsilon_g)}^2)+ O(|\B_g^\epsilon|^{1/2}\|g\|_{L^2(\B)}^2)+O(\|f-1\|_{L^1(\B)}\|g\|_{L^2(\B)}) \nonumber\\
    &+O(\|f-1\|_{L^1(\B)}^2+\|g\|_{L^2(\B)}^3)
\end{align*} 
where $\B_g^\epsilon=\{x\in \B: |g(x)|>\epsilon\}$ for any $0<\epsilon<\pi/2$ and $A_g=\{x\in\B:\cos g(x)\ge 0\}$. 
\end{lemma}

\begin{proof} We use the Taylor expansion from in Lemma \ref{Taylorgen}. 

\begin{align}
\|\widehat{fe^{ig}1_{\B}}\|_q^q  &= \|\widehat{1_{\B}}\|_q^q+q\langle K_q, f\cos g1_{\B}-1_{\B}\rangle \label{bread} -\frac{1}{4}q(q-2)\langle f\sin g1_{\B}* f\sin g1_{\B},L_q\rangle \\
    & +\frac{1}{4}q^2\langle  f\sin g1_{\B}, f\sin g1_{\B}*L_q\rangle+\frac{1}{4}q(q-2)\langle (f\cos g1_{\B}-1_{\B})*(f\cos g1_{\B}-1_{\B}),L_q\rangle \nonumber\\
    & +\frac{1}{4}q^2\langle (f\cos g1_{\B}-1_{\B})*(\tilde{f}\cos \tilde{g}1_{\B}-1_{\B}),L_q\rangle+O(\|fe^{ig}-1\|_{L^{q'}(\B)}^3) \nonumber
\end{align}

Bound the terms with two cosines by $O(\|f\cos g -1\|_{L^1(\B)}^2)$.
Since 
\begin{align*}
    \|f\cos g-1\|_{L^1(\B)}&\le \|f-1\|_{L^1(\B)}+\|\cos g-1\|_{L^1(\B)}\\
    &\le \|f-1\|_{L^1(\B)}+\|g\|_{L^2(\B)}^2,
\end{align*}
we can replace $O(\|f\cos g-1\|_{L^1(\B)}^2)$ by $O(\|f-1\|_{L^1(\B)}^2+\|g\|_{L^2(\B)}^4)$. 

For the terms with sine, first we eliminate the $f$ factor. 

\begin{align*}
\langle  f\sin g1_{\B}, f\sin g1_{\B}*L_q\rangle&=\langle  (f-1+1)\sin g1_{\B}, (f-1+1)\sin g1_{\B}*L_q\rangle \\
&= \langle \sin g 1_{\B},\sin g1_{\B}*L_q\rangle +O(\|(f-1)\sin g\|_{L^1(\B)}\|f\sin g\|_{L^1(\B)})\\
&\le   \langle \sin g 1_{\B},\sin g1_{\B}*L_q\rangle   +O(\|f-1\|_{L^1(\B)}\|g\|_{L^2(\B)}). 
\end{align*}
Next, split the ball into $\B_g^\epsilon$ and $\B\setminus \B_g^\epsilon$. 

\begin{align*}
    \langle \sin g 1_{\B},\sin g1_{\B}*L_q\rangle&= \langle( \sin g1_{\B\setminus \B^\epsilon_g}+\sin g1_{\B^\epsilon_g}),( \sin g1_{\B\setminus \B^\epsilon_g}+\sin g 1_{\B^\epsilon_g})*L_q\rangle \\
    &\le  \langle \sin g1_{\B\setminus \B^\epsilon_g},\sin g1_{\B\setminus \B^\epsilon_g}*L_q\rangle +O(\|\sin g1_{\B_g^\epsilon}\|_1\|\sin g1_{\B}\|_1) \\
    &\le  \langle \sin g1_{\B\setminus \B^\epsilon_g},\sin g1_{\B\setminus \B^\epsilon_g}*L_q\rangle +O(|\B_g^\epsilon|^{1/2}\|g\|_{L^2(\B)}^2). 
\end{align*}
Together, we have

\begin{align*}
    -\frac{1}{4}q(q-2)&\langle f\sin g1_{\B}* f\sin g1_{\B},L_q\rangle  +\frac{1}{4}q^2\langle  f\sin g1_{\B}, f\sin g1_{\B}*L_q\rangle \\
    &\le -\frac{1}{4}q(q-2)\langle \sin g1_{\B\setminus \B_g^\epsilon}* \sin g1_{\B\setminus \B_g^\epsilon},L_q\rangle  +\frac{1}{4}q^2\langle  \sin g1_{\B\setminus \B_g^\epsilon}, \sin g1_{\B\setminus \B_g^\epsilon}*L_q\rangle\\
    &\quad O(|\B_g^\epsilon|^{1/2}\|g\|_{L^2(\B)}^2)+O(\|f-1\|_{L^1(\B)}\|g\|_{L^2(\B)}). 
\end{align*}

Finally, for the term with one cosine, recall $A_g=\{x\in \B:\cos g(x)\ge 0\}$. On the set $\B\setminus \B_g^\epsilon=\{x\in \B:|g(x)|\le \epsilon\}$ where $\epsilon<\pi/2$, we also have $\cos g> 0$, so $\B\setminus \B_g^\epsilon\subset A_g$. Calculate
\begin{align*}
    \langle K_q,f\cos g1_{\B}-1_{\B}\rangle&\le \langle K_q, f\cos g1_{A_g}-1_{\B}\rangle  \\
    &\le \langle K_q, \cos g1_{\B\setminus \B_g^\epsilon}-1_{\B\setminus \B_g^\epsilon}\rangle+\langle K_q,\cos g1_{A_g\cap \B_g^\epsilon}- 1_{A_g\cap \B_g^\epsilon}\rangle\\
    &\quad-\inf_{\B}K_q\cdot |\B_g^\epsilon\setminus A_g|\\
    &\le \langle K_q, \cos g1_{\B\setminus \B_g^\epsilon}-1_{\B\setminus \B_g^\epsilon}\rangle-\inf_{\B}K_q\cdot \left(\|\cos g-1\|_{L^1(A_g\cap \B_g^\epsilon)}+|\B_g^\epsilon\setminus A_g|\right).  
\end{align*}
Using the above analysis in (\ref{bread}), we have
\begin{align}
\|\widehat{fe^{ig}1_{\B}}\|_q^q  &\le  \|\widehat{1_{\B}}\|_q^q+q\langle K_q, \cos g1_{\B\setminus \B_g^\epsilon}-1_{\B\setminus \B_g^\epsilon}\rangle-q\inf_{\B}K_q\cdot \left(\|\cos g-1\|_{L^1(A\cap \B_g^\epsilon)}+|\B_g^\epsilon\setminus A|\right) \label{bread2}\\
    &-\frac{1}{4}q(q-2)\langle \sin g1_{\B\setminus \B_g^\epsilon}* \sin g1_{\B\setminus \B_g^\epsilon},L_q\rangle +\frac{1}{4}q^2\langle  \sin g1_{\B\setminus \B_g^\epsilon}, \sin g1_{\B\setminus \B_g^\epsilon}*L_q\rangle \nonumber \\
    &\quad O(|\B_g^\epsilon|^{1/2}\|g\|_{L^2(\B)}^2)+O(\|f-1\|_{L^1(\B)}\|g\|_{L^2(\B)}) \nonumber\\
    &+O(\|f-1\|_{L^1(\B)}^2+\|g\|_{L^2(\B)}^4)+O(\|fe^{ig}-1\|_{L^{q'}(\B)}^3). \nonumber
\end{align} 
Approximate the remaining trigonometric functions by the following Taylor expansions for $t\in\R$:

\[\sin t=t+O(t^3) \qquad \text{and}\qquad \cos t= 1-\frac{t^2}{2}+O(t^4). \]
This combined with the definition of $\B_g^\epsilon$ gives

\begin{align*}
q\langle &K_q, \cos g1_{B\setminus \B^\epsilon_g}-1_{B\setminus \B^\epsilon_g}\rangle -\frac{1}{4}q(q-2)\langle \sin g1_{\B\setminus \B^\epsilon_g}*\sin g1_{\B\setminus \B^\epsilon_g},L_q\rangle \\
    &\quad +\frac{1}{4}q^2\langle  \sin g1_{\B\setminus \B^\epsilon_g},\sin g1_{\B\setminus \B^\epsilon_g}*L_q\rangle\\
&=q\langle K_q, (1-g^2/2+O(g^4))1_{\B\setminus \B^\epsilon_g}-1_{\B\setminus \B^\epsilon_g}\rangle \\
    &\quad -\frac{1}{4}q(q-2)\langle (g+O(g^3))1_{\B\setminus \B^\epsilon_g}*(g+O(g^3))1_{\B\setminus \B^\epsilon_g},L_q\rangle \\
    &\quad +\frac{1}{4}q^2\langle  (g+O(g^3))1_{\B\setminus \B^\epsilon_g},(g+O(g^3))1_{\B\setminus \B^\epsilon_g}*L_q\rangle\\
&\le -\frac{q}{2}\langle K_q, g^21_{\B\setminus \B^\epsilon_g}\rangle+\epsilon^2O(\|g\|_{L^2(\B\setminus \B^\epsilon_g)}^2) -\frac{1}{4}q(q-2)\langle g1_{\B\setminus \B^\epsilon_g}*g1_{\B\setminus \B^\epsilon_g},L_q\rangle\\
    &\quad+\frac{1}{4}q^2\langle  g1_{\B\setminus \B^\epsilon_g},g1_{\B\setminus \B^\epsilon_g}*L_q\rangle+O(\|g\|_{L^2(\B)}^3)
\end{align*}

Finally, we note that since $q'<3/2$ and $\|fe^{ig}-1\|_{L^{q'}(\B)}\le \|f-1\|_{L^{q'}(\B)}+\|g\|_{L^2(\B)}(2|\B|)^{(2-q')/2q'}$, the error terms may be combined to 
\[ O(\|f-1\|_{L^2(\B)}^2+\|g\|_{L^2(\B)}^3+\|g\|_{L^2(\B)}^4+\|fe^{ig}-1\|_{L^{q'}(\B)}^3)\le  O(\|f-1\|_{L^2(\B)}^2+\|g\|_{L^2(\B)}^3). \]

\end{proof}

\end{subsection}

\begin{subsection}{Connection with a spectral problem.\label{specsec}}

In the previous section, for $fe^{ig}1_{\B}$ with $\|f-1\|_{L^1(\B)}$ and $\|g\|_{L^2(\B)}$ small, we expressed  $\|\widehat{fe^{ig}1_{\B}}\|_q^q$ as $\|\widehat{1_{\B}}\|_q^q$ plus a quadratic form in $g1_{\B\setminus\B_g^\epsilon}$ and a small error. In this section, we  analyze a spectral problem concerning that quadratic form when $q\ge 4$ is an even integer in order to obtain a more descriptive upper bound for $\|\widehat{fe^{ig}1_{\B}}\|_q^q$. 

\begin{definition} Define $T_q:L^2(\B)\to L^2(\B)$ to be the linear operator which is the composition of multiplication by $K_q^{-1/2}$, followed by  convolution with $L_q$, followed by multiplication by $K_q^{-1/2}$. That is, for $h\in L^2(\B)$,
\[ h\overset{T_q}\longmapsto \left.K_q^{-1/2}(K_q^{-1/2}h1_{\B}*L_q)\right|_{\B}.\]
\end{definition}

Observe that $T_q$ is bounded on $L^2(\B)$ since $K_q$ is bounded above and below by positive quantities on $\B$, so
\[ \|K_q^{-1/2}(K^{-1/2}h*L_q)\|_{L^2(\B)}\le \|K_q^{-1/2}\|^2_{L^\infty(\B)}\|L_q\|_{L^1(2\B)}\|h\|_{L^2(\B)} .\]

Using that $\widehat{1_{\B}}$ is a real-valued function that satisfies $\|\widehat{1_{\B}}(\xi)|\le C_d(1+|\xi|)^{-(d+1)/2}$ and that $\widehat{L_q}=|\widehat{1_{\B}}|^{q-2}$, we will show that $T_q$ is a compact operator. Since $K_q$ is bounded above and below by positive constants on $2\B$, multiplication by $K_q^{-1/2}$ defines a bounded operator on $L^2(\B)$ and therefore it suffices to show that convolution with $L_q$ is compact. But convolution with any continuous function, followed by restriction to the ball, defines a compact operator, so $T_q$ is compact.

Finally, since $K_q$ is real and $L_q$ is symmetric, 
\[ \langle T_qh,f\rangle = \langle K_q^{-1/2}h*L_q,K_q^{-1/2}h \rangle= \langle K_q^{-1/2}h,L_q*K_q^{-1/2}h \rangle=\langle h, T_qf\rangle, \]
so $T_q$ is self-adjoint. Let $Q_q$ be the quadratic form on $L^2(\B)$ defined by 
\[Q_q(f,h)=\langle f,T_qh\rangle\]
where $f,h\in L^2(\B)$. As above, $\tilde{f}(x)=f(-x)$. 

\begin{definition} \label{defH}Let $d\ge 1$ and $q\ge 4$ be an even integer. Let $\mc{H}$ denote the subspace of $L^2(\B)$ of functions of the form $K_q^{1/2}(x)(\a\cdot x+b)$ where $\a\in\R^d$ and $b\in\R$. Let $P_{\mc{H}}:L^2(\B)\to L^2(\B)$ denote the orthogonal projection onto $\mc{H}$. 
\end{definition} 

\begin{lemma}\label{speclem}Let $d\ge 1$. Let $q\ge 4$ be an even integer. Then there exists $c>0$ depending on the dimension and $q$ such that

\begin{align} -\frac{q}{2}\|h\|_{L^2(\B)}^2-\frac{1}{4}q(q-2)Q_q(h,\tilde{h}) +\frac{1}{4}q^2Q_q(h,h)\le -c\|(I-P_{\mc{H}})h\|^2_{L^2(\B)}\label{specineq}\end{align}
for every real-valued $h\in L^2(\B)$. 

\end{lemma}

\begin{proof} Since $T_q$ is a compact, self-adjoint linear operator, we can write $L^2(\B)$ as a direct sum of eigenspaces. For a fixed eigenvalue , we can further orthogonally decompose the corresponding eigenspace into even eigenfunctions and odd eigenfunctions since $K_q=\tilde{K}_q$ and if $T_q\p=\lambda\p$, $1/2(\p+\tilde{\p})+1/2(\p-\tilde{\p})$ is a unique  representation of $\p$ as a sum of an even eigenvector with eigenvalue $\lambda$ and an odd eigenvector with eigenvalue $\lambda$. Since $T_q$ can be regarded as an operator on real-valued functions in $L^2(\B)$, we can assume that the eigenfunctions are real-valued. Thus we can expand $h$ as say $ h=\sum_{n=0}^\infty h_n$, where the $h_n$ are pairwise orthogonal eigenfunctions of $T_q$, either even or odd, real-valued, and associated with eigenvalues $\lambda_n$. The spectrum is real and $\lambda_n\to 0$ as $n\to\infty$. Assume that $|\lambda_n|$ is nonincreasing and calculate

\begin{align} 
-\frac{q}{2}\|h\|_{L^2(\B)}^2-&\frac{1}{4}q(q-2)Q_q(h,\tilde{h}) +\frac{1}{4}q^2Q_q(h,h)\nonumber\\
&= \sum_{n\le N}\left(-\frac{q}{2}\|h_n\|_{L^2(\B)}^2-\frac{1}{4}q(q-2)\lambda_n\langle h_n,\tilde{h_n}\rangle +\frac{1}{4}q^2\lambda_n\|h_n\|_2\right)\nonumber \\
&\quad +\sum_{n>N}\left(-\frac{q}{2}\|h_n\|_{L^2(\B)}^2-\frac{1}{4}q(q-2)\lambda_n\langle h_n,\tilde{h_n}\rangle +\frac{1}{4}q^2\lambda_n\|h_n\|_2^2\right) \nonumber\\
&\le  \sum_{n\le N}\left(-\frac{q}{2}\|h_n\|_{L^2(\B)}^2-\frac{1}{4}q(q-2)\lambda_n\langle h_n,\tilde{h_n}\rangle +\frac{1}{4}q^2\lambda_n\|h_n\|^2_2\right) \label{boing1} \\
&\quad +\left(-\frac{q}{2}+\frac{1}{2}q(q-1)|\lambda_N| \right)\sum_{n>N}\|h_n\|_{L^2(\B)}^2\nonumber 
\end{align}
We return to this expression (\ref{boing1}) after understanding the case for a single eigenfunction. Fix an eigenfunction $\p$ of $T_q$ with eigenvalue $\lambda$. We analyze

\begin{align} -\frac{q}{2}\|\p\|_{L^2(\B)}^2-\frac{1}{4}q(q-2)Q_q(\p,\tilde{\p}) +\frac{1}{4}q^2Q_q(\p,\p). \label{specphi}\end{align}

Note that since $K_q$ and $L_q$ are even functions, $T_q\tilde{\p}=\lambda\tilde{\p}$ as well. If $\lambda=0$, then (\ref{specineq}) for $h=\p$ is trivial since 
\begin{align*} 
-\frac{q}{2}\|\p\|_{L^2(\B)}^2&-\frac{1}{4}q(q-2)Q_q(\p,\tilde{\p}) +\frac{1}{4}q^2Q_q(\p,\p)= -\frac{q}{2}\|\p\|_{L^2(\B)}^2 \\ 
&\le -\frac{q}{2}\|(I-P_{\mc{H}})\p\|_{L^2(\B)}^2 .
\end{align*}
Thus we can assume that $\lambda\not=0$. In this case, 
\[ |\lambda||\p(x)|= |K_q^{-1/2}(x)(K^{-1/2}\p*L_q)(x)|\le \|K_q^{-1/2}\|^2_{L^\infty(\B)}\|\p\|_{L^2(\B)}\|L_q\|_{L^2(\B)} ,\]
so $\|\p\|_{L^\infty(\B)}$ is finite. Following \cite{c1} and \cite{N1}, (\ref{specphi}) is analyzed for an eigenfunction $\p$ by considering a Taylor expansion of $\|\widehat{e^{it\s}}\|_q^q$ where $\s=K_q^{-1/2}\p $ and $t\in\R$ is an auxiliary parameter. Choose $t>0$ sufficiently small so that $\cos t\s\ge 0$ on $\B$ and the hypotheses of Lemma \ref{Taylorfreq} are satisfied with $f=1$ and $g=t\s$. Executing the proof of Lemma \ref{Taylorfreq} without expanding the term $\langle K_q,(f\cos(g)-1)1_{\B}\rangle$ yields the following equality:

\begin{align*}
\|\widehat{e^{it\s}1_{\B}}\|_q^q  &=  \|\widehat{1_{\B}}\|_q^q+q\langle K_q, (\cos(t\s)-1)1_{\B}\rangle -\frac{1}{4}q(q-2)\langle t\s1_{\B\setminus \B^\epsilon_{t\s}}*t\s1_{\B\setminus \B^\epsilon_{t\s}},L_q\rangle\\
    & +\frac{1}{4}q^2\langle  t\s1_{\B\setminus \B^\epsilon_{t\s}},t\s1_{\B\setminus \B^\epsilon_{t\s}}*L_q\rangle + O(|\B_{t\s}^\epsilon|^{1/2}t^2\|\s\|_{L^2(\B)}^2+t^3\|\s \|_{L^2(\B)}^3).
\end{align*} 
\noindent where $\epsilon\in(0,\pi/2)$ and $\B_{t\s}^\epsilon=\{x\in\B:t|K^{-1/2}(x)\p(x)|>\epsilon\}$. Since $\|\p\|_{L^\infty(\B)}<\infty$, for $t<\epsilon(\|K^{-1/2}\p\|_{L^\infty(\B)})^{-1}$, the set  $\B_{t\s}^\epsilon$ is empty. Thus the statement we have from the proof of Lemma \ref{Taylorfreq} for $t<\epsilon(\|K^{-1/2}\p\|_{L^\infty(\B)})^{-1}$ is

\begin{align*} 
\|\widehat{e^{it\s}1_{\B}}\|_q^q&= \|\widehat{1_{\B}}\|_q^q-q\langle K_q, (\cos(t\s)-1)1_{\B}\rangle -\frac{1}{4}q(q-2)\langle t\s1_{\B}*t\s1_{\B},L_q\rangle \\
    &+\frac{1}{4}q^2\langle  t\s1_{\B},t\s1_{\B}*L_q\rangle +O(t^3\|\s\|_{L^2(\B)}^3).  
\end{align*}
Now if we expand the cosine, we have the equality

\begin{align} 
\|\widehat{e^{it\s}1_{\B}}\|_q^q&= \|\widehat{1_{\B}}\|_q^q-\frac{q}{2}\langle K_q, (t\s)^21_{\B}\rangle+t^4O(\|\s\|_{L^\infty(\B)}^2\|\s\|_{L^2(\B)}^2) -\frac{1}{4}q(q-2)\langle t\s1_{\B}*t\s1_{\B},L_q\rangle \nonumber \\
    &+\frac{1}{4}q^2\langle  t\s1_{\B},t\s1_{\B}*L_q\rangle +O(t^3\|\s\|_{L^2(\B)}^3)  \nonumber\\
    &=\|\widehat{1_{\B}}\|_q^q-t^2\left[\frac{q}{2}\langle K_q, (\s)^21_{\B}\rangle-\frac{1}{4}q(q-2)\langle \s1_{\B}*\s1_{\B},L_q\rangle +\frac{1}{4}q^2\langle  \s1_{\B},\s1_{\B}*L_q\rangle\right] +O_{\p}(t^3)  \nonumber\\
    &=\|\widehat{1_{\B}}\|_q^q-t^2\left[\frac{q}{2}\|\p\|_{L^2(\B)}^2-\frac{1}{4}q(q-2)Q_q(\p,\tilde{\p}) +\frac{1}{4}q^2Q_q(\p,\p)\right] +O_{\p}(t^3) \nonumber\\
    &=\|\widehat{1_{\B}}\|_q^q-t^2\left[\frac{q}{2}\|\p\|_{L^2(\B)}^2-\frac{\lambda}{4}q(q-2)\langle\p,\tilde{\p}\rangle +\frac{\lambda}{4}q^2\|\p\|_{L^2(\B)}^2 \right] +O_{\p}(t^3) \label{bd1},
\end{align}
\noindent where the final big-$O_\p$ depends on the dimension, the exponent $q$, and on the $L^\infty$ and $L^2$ norms of $\p$. Note that we used that $|\s|$ is bounded above and below by a constant (depending on $q$) multiple of $\p$ on $\B$. Since $q$ is an even integer, $\|\widehat{e^{it\s}1_{\B}}\|_q^q\le \|\widehat{1_{\B}}\|_q^q$ and so
\begin{align}\frac{q}{2}\|\p\|_{L^2(\B)}^2-\frac{\lambda}{4}q(q-2)\langle\p,\tilde{\p}\rangle +\frac{\lambda}{4}q^2\|\p\|_{L^2(\B)}^2 \ge 0 \label{assump2}\end{align}
from (\ref{bd1}). Expressing $q=2m$, we can also write

\begin{align*} 
&\|\vwidehat{e^{it\s}1_{\B}}\|_q^q=\textrm{Re}\int_{\B^{q-1}} e^{it(\psi(x_1)+\cdots +\psi(x_m)-\psi(y_2)-\cdots -\psi(y_m)-\psi(L(x,y)))}1_{\B}(L(x,y)) dxdy 
\end{align*}
\noindent where $x=(x_1,\ldots,x_m)$, $y=(y_2,\ldots,y_m)$, and $L(x,y)=x_1+\cdots +x_m-y_2-\cdots y_m$. 
Let $\a(x,y)=\s(x_3)+\cdots +\psi(x_m)-\psi(y_2)-\cdots -\psi(y_m)-\psi(L(x,y))$ . Then for all sufficiently small $t$, since $\cos(\theta)-1\le -\frac{\theta^2}{4}$ for $|\theta|\le\theta/2$,

\begin{align}
&\|\vwidehat{e^{it\s}1_{\B}}\|_q^q=\|\widehat{1_{\B}}\|_q^q+\|\vwidehat{e^{it\s}1_{\B}}\|_q^q-\|\widehat{1_{\B}}\|_q^q \nonumber   \nonumber\\
&=\|\widehat{1_{\B}}\|_q^q -\int_{B^{q-1}} |\cos(t(\psi(x_1)+\s(x_2)+\a(x,y)))-1|1_{\B}(L(x,y)) dxdy    \nonumber\\
&\le \|\widehat{1_{\B}}\|_q^q-\frac{t^2}{4}\int_{B^{q-1}} (\psi(x_1)+\s(x_2)+\a(x,y))^21_{\B}(L(x,y)) dxdy \label{bd2}.  
\end{align}

Combining (\ref{bd1}) with (\ref{bd2}) gives

\begin{align*}
\frac{t^2}{4} \int_{B^{q-1}} (\psi(x_1)+\s(x_2)&+\a(x,y))^21_{\B}(L(x,y)) dxdy\\
&\quad\le  t^2\left[\frac{q}{2}\|\p\|_{L^2(\B)}^2-\frac{\lambda}{4}q(q-2)\langle\p,\tilde{\p}\rangle +\frac{\lambda}{4}q^2\|\p\|_{L^2(\B)}^2 \right] +O_{\p}(t^3)  
\end{align*}
for all sufficiently small $t>0$. If the coefficient of $t^2$ on the right hand side is 0, then
\[ \int_{B^{q-1}} (\psi(x_1)+\s(x_2)+\a(x,y))^21_{\B}(L(x,y)) dxdy=0  ,  \]
which means that 

\begin{align}
&\|\vwidehat{e^{it\s}1_{\B}}\|_q^q=\|\widehat{1_{\B}}\|_q^q -\int_{B^{q-1}} |\cos(t(\psi(x_1)+\s(x_2)+\a(x,y)))-1|1_{\B}(L(x,y)) dxdy =\|\widehat{1_{\B}}\|_q^q.  \nonumber
\end{align}
By Lemma \ref{extg}, $e^{it\s}=e^{i(\a\cdot x+b)}$ for some $\a\in\R^d$ and $b\in\R$. Thus the inequality (\ref{assump2}) is strict unless $e^{iK^{-1/2}\p}$ takes the form $e^{i(\a\cdot x+b)}$. Using that $\lambda$ is nonzero, we have for each $x\in\B$ the expression 
\[  \p(x)= \lambda^{-1} K_q^{-1/2}(x)(K_q^{-1/2}\p*L_q)(x) .\]
Since $K_q^{-1/2}$ is continuous and $L_q\in L^2(\R^d)$, $\p$ is continuous on $\B$. Note that  $e^{iK_q^{-1/2}\p}=e^{i(\a\cdot x+b)}$ implies that $K_q^{-1/2}\p(x)=\a\cdot x+b+f(x)$ for some function $f:\B\to2\pi\Z$. The only continuous such function is constant, so $\p(x)=K_q^{1/2}(\a\cdot x+b')$ for $b'=b+2\pi n$ for some $n\in\Z$. Conclude that the inequality (\ref{assump2}) is strict unless $\p\in\mc{H}$, where $\mc{H}$ was defined in Definition \ref{defH}.

Finally, we use this in (\ref{boing1}) and conclude that 
there exists $c>0$ so that

\begin{align*} 
-\frac{q}{2}\|h\|_{L^2(\B)}^2-&\frac{1}{4}q(q-2)Q_q(h,\tilde{h}) +\frac{1}{4}q^2Q_q(h,h)\nonumber\\
&\le -c \sum_{\substack{n\le N\\h_n\not\in \mc{H}}}\|h_n\|_{L^2(\B)}^2-c\sum_{n>N}\|h_n\|_{L^2(\B)}^2\\
&\le -c\|(I-P_{\mc{H}})h\|_{L^2(\B)}^2. 
\end{align*}

\end{proof}









\end{subsection}

\subsection{Conclusion of the spectral analysis for $q$ near an even integer and $E=\B$}
Let $K_q,L_q$ be the functions defined in (\ref{Kq}) and (\ref{Lq}). Use the frequency Taylor expansion from Lemma \ref{Taylorfreq} and compare the main terms with $K_q$ and $L_q$ to analogous terms with $K_{\overline{q}}$ and $L_{\overline{q}}$ where $\overline{q}$ is the closest even integer. Then make use of the spectral analysis in Lemma \ref{speclem} to obtain the following theorem. We will use the following theorem in the proof of Proposition \ref{support} in \textsection\ref{supportsec} and the proof of Proposition \ref{freq} in \textsection\ref{freq2}.

\begin{theorem} \label{Taylorfreqcor}Let $d\ge 1$ and let $\overline{q}\ge 4$ be an even integer. There exist $\delta_0,\rho, >0$ all depending on the dimension and $\overline{q}$ as well as $c_{q,d}>0$ so that the following holds. Let $q\in(3,\infty)$, $E\subset\R^d$ be a Lebesgue measurable set with $|E|\le |\B|$, $0\le f\le 1$, and $g$ be real valued. If $|q-\overline{q}|<\rho$, $\|f-1\|_{L^1(\B)}\le \delta_0$, $\|g\|_{L^2(\B)}\le \delta_0$, and $|g|\le\frac{5\pi}{4}$, then 

\begin{align*}
\|\widehat{fe^{ig}1_{\B}}\|_q^q&\le  \|\widehat{1_{\B}}\|_q^q -c_{q,d}\|(I-P_{\mc{H}})K_{\overline{q}}^{1/2}g\|_{L^2(\B)}^2  +o_{q-\overline{q}}(1)\|g\|_{L^2(\B)}^2 \nonumber \\ 
    &\quad+O(\|f-1\|_{L^1(\B)}\|g\|_{L^2(\B)})+O(\|f-1\|_{L^1(\B)}^2+\|g\|_{L^2(\B)}^{5/2}) . \nonumber
\end{align*} 
\end{theorem}

\begin{proof} The function $fe^{ig}1_{\B}$ satisfies the hypotheses of Lemma \ref{Taylorfreq}. We have the expansion

\begin{align}
\|\widehat{fe^{ig}1_{\B}}\|_q^q  &\le  \|\widehat{1_{\B}}\|_q^q-q\inf_{\B}K_q\cdot \left(\|\cos g-1\|_{L^1(A_g\cap \B_g^\epsilon)}+|\B_g^\epsilon\setminus A_g|\right)  \label{fra1} \\
    &-\frac{q}{2}\langle K_q, g^21_{\B\setminus \B^\epsilon_g}\rangle -\frac{1}{4}q(q-2)\langle g1_{\B\setminus \B^\epsilon_g}*g1_{\B\setminus \B^\epsilon_g},L_q\rangle+\frac{1}{4}q^2\langle  g1_{\B\setminus \B^\epsilon_g},g1_{\B\setminus \B^\epsilon_g}*L_q\rangle\nonumber \\
    & +\epsilon^2O(\|g\|_{L^2(\B\setminus \B^\epsilon_g)}^2)+ O(|\B_g^\epsilon|^{1/2}\|g\|_{L^2(\B)}^2)+O(\|f-1\|_{L^1(\B)}\|g\|_{L^2(\B)}) \nonumber\\
    &+O(\|f-1\|_{L^1(\B)}^2+\|g\|_{L^2(\B)}^3) \nonumber
\end{align}

We analyze the three main terms in the expansion:

\begin{align} 
-\frac{q}{2}\langle K_q, &g^21_{\B\setminus \B^\epsilon_g}\rangle -\frac{1}{4}q(q-2)\langle g1_{\B\setminus \B^\epsilon_g}*g1_{\B\setminus \B^\epsilon_g},L_q \rangle +\frac{1}{4}q^2\langle  g1_{\B\setminus \B^\epsilon_g},g1_{\B\setminus \B^\epsilon_g}*L_q\rangle \nonumber\\
&=-\frac{q}{2}\langle K_q-K_{\overline{q}}+K_{\overline{q}}, g^21_{\B\setminus \B^\epsilon_g}\rangle -\frac{1}{4}q(q-2)\langle g1_{\B\setminus \B^\epsilon_g}*g1_{\B\setminus \B^\epsilon_g},L_q-L_{\overline{q}}+L_{\overline{q}} \rangle \nonumber\\
    &\quad +\frac{1}{4}q^2\langle  g1_{\B\setminus \B^\epsilon_g},g1_{\B\setminus \B^\epsilon_g}*(L_q-L_{\overline{q}}+L_{\overline{q}})\rangle \nonumber\\
&\le -c_{\overline{q},d} \|(I-P_{\mc{H}})K_{\overline{q}}^{1/2}g1_{\B\setminus\B_g^\epsilon}\|_{L^2(\B)}^2+o_{q-\overline{q}}(1)\|g\|_{L^2(\B\setminus \B_g^\epsilon)}^2  \label{d4} 
\end{align}
where we use Lemma \ref{speclem} and Lemma \ref{Kqlem2} in (\ref{d4}). Using this in (\ref{fra1}) gives

\begin{align}
\|\widehat{fe^{ig}1_{\B}}\|_q^q  &\le  \|\widehat{1_{\B}}\|_q^q-q\inf_{\B}K_q\cdot \left(\|\cos g-1\|_{L^1(A_g\cap \B_g^\epsilon)}+|\B_g^\epsilon\setminus A_g|\right)  \label{fra*1} \\
    &-c_{\overline{q},d}\|(I-P_{\mc{H}})K_{\overline{q}}^{1/2}g1_{\B\setminus \B_g^\epsilon}\|_{L^2(\B)}^2 +o_{q-\overline{q}}(1)\|g\|_{L^2(\B\setminus\B_g^\epsilon)}^2 \nonumber \\
    & +\epsilon^2O(\|g\|_{L^2(\B\setminus \B^\epsilon_g)}^2)+ O(|\B_g^\epsilon|^{1/2}\|g\|_{L^2(\B)}^2)+O(\|f-1\|_{L^1(\B)}\|g\|_{L^2(\B)}) \nonumber\\
    &+O(\|f-1\|_{L^1(\B)}^2+\|g\|_{L^2(\B)}^3) \nonumber.
\end{align}

Since $|g|\le \frac{5\pi}{4}$, we can combine the $ \|(I-P_{\mc{H}})K_{\overline{q}}^{1/2}g1_{\B\setminus\B_g^\epsilon}\|_{L^2(\B)}^2$ above with the other negative term above as follows. Choose $c_0>0$ so that $1-\cos \theta\ge c_0\theta^2$ for $|\theta|\le \frac{5\pi}{4}$. Then 

\begin{align}
    \|\cos g-1\|_{L^1(A_g\cap \B_g^\epsilon)}&+|\B_g^\epsilon\setminus A_g|+ \|(I-P_{\mc{H}})K_{\overline{q}}^{1/2}g1_{\B\setminus \B^\epsilon_g}\|_{L^2(\B)}^2 \ge  c_0\|g\|_{L^2(A_g\cap \B_g^\epsilon)}^2 \nonumber\\
    &\qquad +\frac{16}{25\pi^2}\|g\|_{L^2(\B_g^\epsilon\setminus A_g)}^2+ \|(I-P_{\mc{H}})K_{\overline{q}}^{1/2}g1_{\B\setminus \B^\epsilon_g}\|_{L^2(\B)}^2\nonumber \\
    &\ge  c_0\|K_{\overline{q}}\|_{L^\infty(\B)}^{-1}\|K_{\overline{q}}^{1/2}g\|_{L^2(A_g\cap \B_g^\epsilon)}^2 \nonumber +\frac{16}{25\pi^2}\|K_{\overline{q}}\|_{L^\infty(\B)}^{-1}\|K_{\overline{q}}^{1/2}g\|_{L^2(\B_g^\epsilon\setminus A_g)}^2\\
    &\qquad+ \|(I-P_{\mc{H}})K_{\overline{q}}^{1/2}g1_{\B\setminus \B^\epsilon_g}\|_{L^2(\B)}^2\nonumber \\
    &\ge C_0 \|K_{\overline{q}}^{1/2}g1_{\B_g^\epsilon}+(I-P_{\mc{H}})K_{\overline{q}}^{1/2}g1_{\B\setminus \B^\epsilon_g}\|_{L^2(\B)}^2 \nonumber\\
&= C_0 \|K_{\overline{q}}^{1/2}g-P_{\mc{H}}(K_{\overline{q}}^{1/2}g1_{\B\setminus \B^\epsilon_g})\|_{L^2(\B)}^2 \ge  C_0 \|K_{\overline{q}}^{1/2}g-P_{\mc{H}}(K_{\overline{q}}^{1/2}g)\|_{L^2(\B)}^2 \label{fra2}
\end{align}
for an appropriate constant $C_0>0$. So we have for another constant $c>0$

\begin{align}
\|\widehat{fe^{ig}1_{\B}}\|_q^q&\le \|\widehat{1_{\B}}\|_q^q-c \|(I-P_{\mc{H}})K_{\overline{q}}^{1/2}g\|_{L^2(\B)}^2 \label{fra*2} \\
    &+o_{q-\overline{q}}(1)\|g\|_{L^2(\B\setminus\B_g^\epsilon)}^2  +\epsilon^2O(\|g\|_{L^2(\B\setminus \B^\epsilon_g)}^2)\nonumber \\
    &+ O(|\B_g^\epsilon|^{1/2}\|g\|_{L^2(\B)}^2)+O(\|f-1\|_{L^1(\B)}\|g\|_{L^2(\B)}) \nonumber\\
    &+O(\|f-1\|_{L^1(\B)}^2+\|g\|_{L^2(\B)}^3) \nonumber.
\end{align}

Use the bound $|\B_g^\epsilon|^{1/2}\le \epsilon^{-1}\|g\|_{L^2(\B)}$ and choose $\epsilon=\|g\|_{L^2(\B)}^{1/2} $ to simplify the above to 

\begin{align}
\|\widehat{fe^{ig}1_{\B}}\|_q^q&\le  \|\widehat{1_{\B}}\|_q^q -c\|(I-P_{\mc{H}})K_{\overline{q}}^{1/2}g\|_{L^2(\B)}^2  \label{s5*}+o_{q-\overline{q}}(1)\|g\|_{L^2(\B)}^2 \\ 
    &\quad+O(\|f-1\|_{L^1(\B)}\|g\|_{L^2(\B)})+O(\|f-1\|_{L^1(\B)}^2+\|g\|_{L^2(\B)}^{5/2}) . \nonumber
\end{align}

\end{proof}

\section{Mostly support variation: $MN|E\Delta\B|\ge \max(N\|g\|_{L^2(E)},\|f-1\|_{L^1(E)}^{1/2})$\label{supportsec}}

Let $K_q,L_q$ be the functions defined in (\ref{Kq}) and (\ref{Lq}). We employ the more detailed Taylor expansion from Lemma \ref{Taylorfg} in this section.

\begin{proposition}\label{support}
Let $d\ge 1$ and let ${\overline{q}}\ge 4$ be an even integer and $M,N\in\R^+$. There exists $\delta_0=\delta_0({\overline{q}},d,M,N)>0$ and $\rho=\rho(\delta_0,{\overline{q}},M,N)>0$ such that the following holds. Let $q\in(3,\infty)$, $E\subset\R^d$ be a Lebesgue measurable set with $|E|\le |\B|$, $0\le f\le 1$, and $g$ be real valued. Suppose that  $\|f-1\|_{L^1(\B)}\le\delta_0$,  $\|g\|_{L^2(E)}\le \delta_0$, $|E\Delta\B|\le2 \text{dist}(E,\mathfrak{E})\le\delta_0$, and $|q-{\overline{q}}|\le\rho$. If
\[ MN|E\Delta\B|\ge \max(N\|g\|_{L^2(E)},\|f-1\|_{L^1(E)}^{1/2}), \]
then 
\begin{align*}
\|\widehat{fe^{ig}1_{\B}}\|_q^q&\le  \|\widehat{1_{\B}}\|_q^q-c_{q,d} \text{dist}(E,\mathfrak{E})^2
\end{align*} 
for a constant $c_{q,d}>0$ depending only on the exponent $q$ and on the dimension. 
\end{proposition} 

\begin{proof}
We begin with the expression from  Lemma \ref{Taylorfg}, in which the terms with $f$ and $g$ are separated from terms with just the support $E$. 
Recall that $g'=g$ on $E\cap \B$ and $g'=0$ on $B\setminus E$ and that $f'=f$ on $E\cap \B$ and $f'=1$ on $B\setminus E$. We have 

\begin{align} 
\|\widehat{fe^{ig}1_E}\|_q^q&=\|\widehat{1_E}\|_q^q+q\langle K_q, f\cos g1_{E\setminus \B}-1_{E\setminus \B}\rangle-\|\widehat{1_{\B}}\|_q^q+\|\widehat{f'e^{ig'}1_{\B}}\|_q^q\label{where} \\
    &\quad+ O((\|g\|_2^2+\|f-1\|_{L^1(E)})|E\Delta\B|^{1/2})+O(\|g\|_{L^2(E)}^3+\|f-1\|^{2}_{L^1(E)}+|E\Delta\B|^{3/q'})\nonumber .
\end{align} 

We use Christ's Theorem 2.6 from \cite{c2} to bound $\|\widehat{1_E}\|_q$:
\[\|\widehat{1_E}\|_q^q\le \|\widehat{1_{\B}}\|_q^q-c_{q,d}|E\Delta\B|^2 . \]

By Theorem \ref{Taylorfreqcor}, we control $\|\widehat{f'e^{ig'}1_{\B}}\|_q$ as follows. 

\begin{align*}
\|\widehat{f'e^{ig'}1_{\B}}\|_q^q&\le  \|\widehat{1_{\B}}\|_q^q -c_{q,d}\|(I-P_{\mc{H}})K_{\overline{q}}^{1/2}g'\|_{L^2(\B)}^2  +o_{q-\overline{q}}(1)\|g'\|_{L^2(\B)}^2 \nonumber \\ 
    &\quad+O(\|f'-1\|_{L^1(\B)}\|g'\|_{L^2(\B)})+O(\|f'-1\|_{L^1(\B)}^2+\|g'\|_{L^2(\B)}^{5/2}) \nonumber \\
&\le \|\widehat{1_{\B}}\|_q^q+0  +o_{q-\overline{q}}(1)M^2|E\Delta\B|^2 \nonumber \\ 
    &\quad+O_{M,N}(|E\Delta\B|^3)+O_{M,N}(\|E\Delta\B|^{4}+|E\Delta\B|^{5/2}) \\
&= \|\widehat{1_{\B}}\|_q^q+o_{q-\overline{q}}(1)M^2|E\Delta\B|^2  +O_{M,N}(|E\Delta\B|^{5/2}) .  
\end{align*} 
Finally, to bound the inner product term from (\ref{where}), using that $K_{\overline{q}}\ge 0$, calculate
\begin{align*}
\langle K_q,f\cos g1_{E\setminus\B}-1_{E\setminus\B}\rangle &= \langle K_q-K_{\overline{q}}+K_{\overline{q}},f\cos g1_{E\setminus\B}-1_{E\setminus\B}\rangle \\
&\le o_{q-\overline{q}}(1)\|f\cos g-1\|_{L^1(E\setminus\B}+\langle K_{\overline{q}},(f\cos g-1)1_{E\setminus\B}\rangle \\
&\le o_{q-\overline{q}}(1)(\|f-1\|_{L^1(E\setminus\B)}+\|\cos g-1\|_{L^1(E\setminus\B})+0 \\
&\le o_{q-\overline{q}}(1)(M^2N^2|E\Delta\B|^2+M^2|E\Delta\B|^2). 
\end{align*}
Using the above bounds in (\ref{where}) gives
\begin{align*} 
\|\widehat{fe^{ig}1_E}\|_q^q&\le \|\widehat{1_{\B}}\|_q^q-c_{q,d}|E\Delta\B|^2+M^2N^2o_{q-\overline{q}}(1)|E\Delta \B|^2\\
    &\quad+O_{M,N}(|E\Delta\B|^{5/2}+|E\Delta\B|^{3/q'}) 
\end{align*} 
If $\delta_0,\rho$ are chosen sufficiently small, then we have the desired result.

\end{proof}

\section{Mostly frequency variation: $\max(\|f-1\|^{1/2}_1,MN|E\Delta\B|)\le N\|g\|_{L^2(E)}$ \label{freq2}}

Let $K_q,L_q$ be the functions defined in (\ref{Kq}) and (\ref{Lq}). As in \textsection\ref{supportsec}, we employ Lemma \ref{Taylorfg} to analyze the contributions from the frequency $g$.

\begin{proposition}\label{freq}
Let $d\ge 1$ and let ${\overline{q}}>3$ be an even integer and $N\in\R^+$. There exist  $\delta_0=\delta_0({\overline{q}},d)>0$, $\rho(\delta_0,{\overline{q}},N)>0$, and $M=M({\overline{q}},\rho)\in\N$,  such that the following holds. Let $q\in(3,\infty)$, $E\subset\R^d$ be a Lebesgue measurable set with $|E|\le |\B|$, $0\le f\le 1$, and $-\pi\le g\le \pi$.  Suppose that  $\|f-1\|_{L^1(\B)}\le\delta_0$,  $\|g\|_{L^2(E)}\le \delta_0$, $ \|e^{ig}-1\|_{L^2(E\cap\B)} \le 2\inf_{\substack{L\text{ affine}\\\R-\text{valued}}}  \|e^{i(g-L)}-1\|_{L^2(E\cap\B)}$, $|E\Delta\B|\le\delta_0$, $|E|= |\B|$, and $|q-{\overline{q}}|\le\rho(\delta_0,{\overline{q}})$. If 
\[\max(\|f-1\|^{1/2}_1,MN|E\Delta\B|)\le N\|g\|_{L^2(E)}, \]
then 
\begin{align*}
\|\widehat{fe^{ig}1_{E}}\|_q^q&\le  \|\widehat{1_{\B}}\|_q^q-c_{q,d}\inf_{\substack{L\text{ affine}\\\R-\text{valued}}}  \|e^{i(g-L)}-1\|_{L^2(E)}
\end{align*} 
for a constant $c_{q,d}>0$ depending only on the exponent $q$ and on the dimension. 

\end{proposition}

For use in the subsequent proof of Proposition \ref{freq}, we state a version of Lemma 4.1 from \cite{c2} with the special case $\eta=1$, noting that $q_d$ in the statement may be taken to be equal to 3.

\begin{lemma}\label{christlem2}\cite{c2} Let $d\ge 1$ and ${\overline{q}}\ge 4$ be an even integer. There exists $\delta_0=\delta_0({\overline{q}})>0$ and $c,C,\rho,\a\in\R^+$ with the following property. Let $E\subset\R^d$ be a Lebesgue measurable set satisfying $|E|=|\B|$. If $|q-{\overline{q}}|<\rho$ and $|E\Delta\B|\le \delta_0$, then 
\[ \|\widehat{1_{E}}\|_q^q\le\|\vwidehat{1_{E\cap 2\B}}\|_q^q-c|E\setminus 2\B|+C|E\Delta\B|\cdot|E\setminus2\B|+C|E\Delta\B|^{2+\a}.  \]
\end{lemma}

\begin{proof}(of Proposition \ref{freq})

Use the expression from Lemma \ref{Taylorfg} in which the terms with $f$ and $g$ are separated from the terms with only the support $E$. Majorize the big-O terms with $\|f -1\|_{L^1(E)}$ or $|E\Delta\B|$ by terms with $\|g\|_{L^2(E)}$. 
\begin{align} 
\|\widehat{f  e^{ig}1_E}\|_q^q=\|\widehat{1_E}\|_q^q+q&\langle K_q, f \cos g1_{E\setminus \B}-1_{E\setminus \B}\rangle-\|\widehat{1_{\B}}\|_q^q \label{fra4} +\|\widehat{f 'e^{ig'}1_{\B}}\|_q^q\\ 
+& O_N(\|g\|_{L^2(E)}^{5/2}+\|g\|_{L^2(E)}^{3/q'})\nonumber
\end{align} 
where $f'=f$ on $E\cap \B$ and $f'=1$ on $\B\setminus E$ and $g'=g$ on $E\cap \B$and $g'=0$ on $\B\setminus E$. We further analyze  $\|\widehat{1_{E}}\|_q^q$ and $\langle K_q,f\cos g1_{E\setminus\B}-1_{E\setminus\B}\rangle$.

Use Lemma \ref{christlem2} to extract $-|E\setminus2\B|$ from $\|\widehat{1_E}\|_q^q$:

\begin{align*}
    \|\widehat{1_E}\|_q^q&\le\|\widehat{1_{E\cap{2\B}}}\|_q^q-c|E\setminus2\B|+C|E\Delta\B|\cdot|E\setminus 2\B|+C|E\Delta\B|^{2+\a} \\
    &\le \|\widehat{1_{\B}}\|_q^q-c|E\setminus 2\B|+\frac{C}{M^2}\|g\|_2^2+O(\|g\|_2^{2+\a})\\
    &\le \|\widehat{1_{\B}}\|_q^q-c/4\|e^{ig}-1\|_{L^2(E\setminus 2\B)}^2+\frac{C}{M^2}\|g\|_2^2+O(\|g\|_2^{2+\a}). 
\end{align*}

As above, let $\overline{q}$ denote the nearest even integer to $q$.  Next bound the term $\langle K_q,f\cos g1_{E\setminus \B}-1_{E\setminus\B}\rangle$ above by a negative multiple of $\|e^{i g}-1\|_{L^2((E\cap2\B)\setminus\B}^2$ plus an error term. Let $A_g=\{x\in E\setminus \B:\cos g\ge 0\}$.

\begin{align}
\langle K_q-K_{\overline{q}}+K_{\overline{q}},& f\cos g1_{E\setminus \B}-1_{E\setminus\B}\rangle = o_{q-\overline{q}}(1)\|f\cos g-1\|_{L^1(E\setminus\B)}+\langle K_{\overline{q}}, f\cos g1_{E\setminus \B}-1_{E\setminus\B}\rangle \nonumber
\\
&\le o_{q-\overline{q}}(1)\|f\cos g-1\|_{L^1(E\setminus\B)}+\langle K_{\overline{q}}, f\cos g1_{A_g}-1_{E\setminus\B}\rangle \nonumber\\ 
&\le o_{q-\overline{q}}(1)(\|f-1\|_{L^1(E)}+\|g\|_{L^2(E)}^2)+\langle K_{\overline{q}}, \cos g1_{A_g}-1_{E\setminus\B}\rangle \nonumber\\ 
&\le o_{q-\overline{q}}(1)(N^2+1)\|g\|_{L^2(E\setminus\B)}^2+\langle K_{\overline{q}}, f\cos g1_{A_g}-1_{A_g}\rangle-\langle K_{\overline{q}},1_{(E\setminus\B)\setminus A_g}\rangle \nonumber\\ 
&\le o_{q-\overline{q}}(1)(N^2+1)\|g\|_{L^2(E\setminus\B)}^2-\inf_{2\B}K_{\overline{q}}\cdot \left(\|\cos g-1\|_{L^1(A_g\cap 2\B)}+ |(E\cap 2\B)\setminus (\B\cup A_g)|\right) \nonumber  \\
&\le o_{q-\overline{q}}(1)(N^2+1)\|g\|_{L^2(E\setminus\B)}^2-\inf_{2\B}K_{\overline{q}}\cdot \|e^{ig}-1\|_{L^2((E\cap 2\B)\setminus\B)}^2 \label{fra7} 
\end{align} 
where we used that $K_{\overline{q}}\ge 0$ everywhere.

By Theorem \ref{Taylorfreqcor}, since $g$ is real-valued with $|g|\le \pi$,

\begin{align*}
\|\widehat{f'e^{ig'}1_{\B}}\|_q^q&\le  \|\widehat{1_{\B}}\|_q^q -c_{q,d}\|(I-P_{\mc{H}})K_{\overline{q}}^{1/2}g'\|_{L^2(\B)}^2  +o_{q-\overline{q}}(1)\|g'\|_{L^2(\B)}^2 \nonumber \\ 
    &\quad+O(\|f'-1\|_{L^1(\B)}\|g'\|_{L^2(\B)})+O(\|f'-1\|_{L^1(\B)}^2+\|g'\|_{L^2(\B)}^{5/2})\nonumber\\
&\le  \|\widehat{1_{\B}}\|_q^q -c_{q,d}\|(I-P_{\mc{H}})K_{\overline{q}}^{1/2}g1_{E\cap \B}\|_{L^2(\B)}^2  +o_{q-\overline{q}}(1)\|g\|_{L^2(E)}^2 \nonumber \\
    &\quad+O_N(\|g\|_{L^2(E)}^{5/2}). \nonumber
\end{align*}
Recalling the definition of $\mc{H}$ in Definition \ref{defH} and the hypotheses about $g$, note that 
\begin{align*}
    \|(I-P_{\mc{H}})K_{\overline{q}}^{1/2}g1_{E\cap \B}\|_{L^2(\B)}&\ge \inf_{\B}K_{\overline{q}}^{1/2}\cdot \|g-K_{\overline{q}}^{-1/2}P_{\mc{H}}(K_{\overline{q}}^{1/2}g1_{E\cap \B})\|_{L^2(E\cap\B)}\\
    &\ge \inf_{\B}K_{\overline{q}}^{1/2}\cdot \|e^{ig}-e^{iK_{\overline{q}}^{-1/2}P_{\mc{H}}(K_{\overline{q}}^{1/2}g1_{E\cap \B})}\|_{L^2(E\cap\B)}\\
    &\ge \inf_{\B}K_{\overline{q}}^{1/2}\cdot \inf_{\substack{L\text{ affine}\\\R-\text{valued}}}  \|e^{i(g-L)}-1\|_{L^2(E\cap\B)} \\
    &\ge \frac{1}{2}\inf_{\B}K_{\overline{q}}^{1/2}\cdot   \|e^{ig}-1\|_{L^2(E\cap\B)}. 
\end{align*}
Combining the above analysis yields

\begin{align} 
\|\widehat{f e^{ig}1_E}\|_q^q&\le \|\widehat{1_{\B}}\|_q^q-c/4\|e^{ig}-1\|_{L^2(E\setminus2\B)}^2+\frac{C}{M^2}\|g\|_{L^2(E)}^2+o_{q-\overline{q}}(1)(N^2+1)\|g\|_{L^2(E\setminus\B)}^2\nonumber\\
    &-\inf_{2\B}K_{\overline{q}}\cdot \|e^{ig}-1\|_{L^2((E\cap 2\B)\setminus\B)}^2\nonumber \\
    &-\frac{c_{q,d}}{4}\inf_{\B}K_{\overline{q}}\cdot \|e^{ig}-1\|_{L^2(E\cap \B)}^2  +o_{q-\overline{q}}(1)\|g\|_{L^2(E)}^2 +O_N(\|g\|_{L^2(E)}^{2+\epsilon}) \nonumber \\
&= \|\widehat{1_{\B}}\|_q^q-\tilde{c}\|e^{ig}-1\|_{L^2(E)}^2\label{this} \\
    &+\left(\frac{C}{M^2}+o_{q-\overline{q}}(1)N^2\right)\|g\|_{L^2(E)}^2+O_N(\|g\|_{L^2(E)}^{2+\epsilon}) \nonumber 
\end{align} 
where $2+\epsilon=\min(2+\a,5/2,3/q')$ and $\tilde{c}>0$ depends on $\overline{q}$ and $d$. Since $|g|\le \pi$, $|e^{ig}-1|^2\ge \pi^{-2}g^2$ almost everywhere on $E$. Thus for $\delta_0$ and $\rho$ sufficiently small depending on $N$ and $M$ small enough depending on $\overline{q}$, 

\begin{align*} 
\|\widehat{f e^{ig}1_E}\|_q^q&= \|\widehat{1_{\B}}\|_q^q-\frac{\tilde{c}}{2}\|e^{ig}-1\|_{L^2(E)}^2 , 
\end{align*} 
which proves the proposition.

\end{proof}


\begin{subsection}{\label{opt22}Optimality of the $L^2$ norm and the exponent $2$}

We show that the exponent $2$ and the $L^2$ norm in the $\inf_{\substack{L\text{ affine}\\\R-\text{valued}}}  \|e^{i(g-L)}-1\|_{L^2(\B)}^2$ from Theorem \ref{mainthm} are optimal in the following lemma and proposition. First we prove a technical sublemma. 

\begin{notation} Let $d\ge 1$ and let $\B=\{x\in\R^d:|x|\le 1\}$. Let the function $R:L^2(\B)\to L^2(\B)$ be defined on real-valued functions $f$ in $L^2(\B)$ by $R(f)(x)= f(x)\in R/(2\pi)$ and $R(f)(x)\in[-\pi,\pi)$. 
\end{notation}

\begin{sublemma} \label{techopt}Let $d\ge 1$, $p\ge 1$. There exists $\epsilon=\epsilon(p,d)>0$  such that if $g:\R^d\to\R$ satisfies $|g|\le \epsilon$, then 
\[ \inf_{\substack{L\text{ affine}\\\R-\text{valued}}}  \|e^{i(g-L)}-1\|_{L^p(\B)}\ge \frac{1}{2}\inf_{\substack{L\text{ affine}\\\R-\text{valued}}}  \|g-L\|_{L^p(\B)}.   \]
\end{sublemma}

\begin{proof} Using the notation $R$ above and that $|e^{i\theta}-1|\ge \frac{1}{2}\theta$ for all $\theta\in[-\pi,\pi)$, note 
\begin{align*}
    \inf_{\substack{L\text{ affine}\\\R-\text{valued}}}  \|e^{i(g-L)}-1\|_{L^p(\B)}&=\inf_{\substack{L\text{ affine}\\\R-\text{valued}}}  \|e^{iR(g-L)}-1\|_{L^p(\B)}\\
    &\ge \frac{1}{2}\inf_{\substack{L\text{ affine}\\\R-\text{valued}}}  \|R(g-L)\|_{L^p(\B)}. 
\end{align*}
By the definition of $R$, 
\[ \inf_{\substack{L\text{ affine}\\\R-\text{valued}}}  \|R(g-L)\|_{L^p(\B)}\le \inf_{\substack{L\text{ affine}\\\R-\text{valued}}}  \|g-L\|_{L^p(\B)}. \]
For the reverse inequality, it suffices to note that
\[ \inf_{\substack{L\text{ affine}\\\R-\text{valued}}}  \|g-L\|_{L^p(\B)}= \inf_{\substack{L:\B\to\R\,\text{ affine}\\ |L|\le 3}}  \|g-L\|_{L^p(\B)} \]
since for $|L|\le 3$ and $\epsilon<1$, $R(g-L)=g-L$. 
Indeed, suppose for $L_0(x)=x\cdot\a+b$ with $\a\in\R^d$, $b\in\R$ that
\begin{align}\|g-L_0\|_{L^p(\B)}\le 2 \inf_{\substack{L\text{ affine}\\\R-\text{valued}}}  \|g-L\|_{L^p(\B)}\le 2\epsilon|\B|^{1/p}.\label{techeasy}
\end{align}

Suppose for $x\in\B$ that $|L_0(x)|\ge 3$. Since $|g|\le \epsilon$, we know that there exists some $y\in\B$ such that $|L_0(y)|\le \epsilon^{1/2}$. Then if $\epsilon<1$,
\begin{align*}
    2\le 3-\epsilon^{1/2}\le |L_0(x)-L_0(y)|=|\a\cdot (x-y)|\le 2|\a|, 
\end{align*}
so $|L_0(z)|\ge 2$ on $S:= \{z\in\B:|z-y|<1\}$. Thus 
\[ \|g-L_0\|_{L^p(\B)}\ge \|g-L_0\|_{L^p(S)}\ge (2-\epsilon)|S|^{1/p}\ge |S|^{1/p}. \]
Since $|S|\ge |\B\cap(\B+e_1)|$ where $e_1=(1,0,\ldots,0)\in\R^d$, if $\epsilon<\frac{|\B\cap(\B+e_1)|^{1/p}}{2|\B|^{1/p}}$, this contradicts (\ref{techeasy}).

\end{proof}

\begin{lemma} \label{optpow} Let $d\ge 1$, $p\ge 1$, $n>0$, and  $\overline{q}\ge 4$ an even integer. There exists $\rho=\rho(\overline{q},d)>0$ such that the following holds. If for some $q>3$ with $|q-\overline{q}|<\rho$, there exists $c_{q,d}>0$ such that 

\begin{align*} 
\|\widehat{e^{ig}1_{\B}}\|_q^q &\le    \|\widehat{1_{\B}}\|_q^q-c_{\overline{q},d}\inf_{\substack{L\text{ affine}\\\R-\text{valued}}}  \|e^{i(g-L)}-1\|_{L^p(\B)}^n
\end{align*}
for any function $g:\R^d\to\R$, then $n\ge 2$.

\end{lemma}

\begin{proof} Let $g\in L^2(\B)\cap L^\infty(\B)$ be a real-valued function. By Lemma \ref{Taylorgen}, for sufficiently small $t>0$ and for $q'$ the conjugate exponent to $q$, 

\begin{align*}
\|\widehat{e^{itg}1_{\B}}\|_q^q    &= \|\widehat{1_{\B}}\|_q^q+q\langle K_q, (\cos (tg)-1)1_{\B}\rangle-\frac{1}{4}q(q-2)\langle \sin (tg)1_{\B}*\sin (tg)1_{\B},L_q\rangle\\
&\quad +\frac{1}{4}q^2\langle \sin (tg)1_{\B},\sin (tg)1_{\B}*L_q\rangle+O(\|\cos (tg)-1\|_{L^1(\B)}^2)+O(\|e^{itg}-1\|_{L^{q'}(\B)}^3) .
\end{align*} 
Since $2|\cos (\theta)-1|= |e^{i\theta}-1|^2$ for $\theta\in\R$, $\|\cos (tg)-1\|_{L^1(\B)}^2\le \|e^{itg}-1\|_{L^2(\B)}^4\le t^4\|g\|_{L^2(\B)}^4$. Since $1<q'<2$, by H\"{o}lder's inequality,$ \|e^{itg}-1\|_{L^{q'}}\le \|e^{itg}-1\|_{L^2(\B)}(2|\B|)^{(2-q')/(2q')}$. Thus we can replace the big-O terms by $O_g(t^3)$.

Combining the above with our hypothesis and rearranging, we have for a constant $C_{q,d}>0$

\begin{align*}
c_{q,d}\inf_{\substack{L\text{ affine}\\\R-\text{valued}}}  \|e^{i(tg-L)}-1\|_{L^p(\B)}^n +O_g(t^3) &\le 
q\langle K_q, (\cos (tg)-1)1_{\B}\rangle-\frac{1}{4}q(q-2)\langle \sin (tg)1_{\B}*\sin (tg)1_{\B},L_q\rangle\\
&\quad +\frac{1}{4}q^2\langle \sin (tg)1_{\B},\sin (tg)1_{\B}*L_q\rangle \\
&\le C_{q,d}t^2\|g\|_{L^2(\B)}^2. 
\end{align*}
Let $\epsilon>0$ be as in Sublemma \ref{techopt}. For $0<t<\epsilon\|g\|_{L^\infty(\B)}^{-1}$, we then have 

\begin{align*}
\frac{c_{q,d}}{2}\inf_{\substack{L\text{ affine}\\\R-\text{valued}}}  \|tg-L\|_{L^p(\B)}^n+O_g(t^3)&\le c_{q,d}\inf_{\substack{L\text{ affine}\\\R-\text{valued}}}  \|e^{i(tg-L)}-1\|_{L^p(\B)}^n +O_g(t^3) \\
&\le C_{q,d}t^2\|g\|_{L^2(\B)}^2. 
\end{align*}
Divide by $t^2$ to get

\begin{align*}
\frac{c_{q,d}}{2}t^{n-2}\inf_{\substack{L\text{ affine}\\\R-\text{valued}}}  \|g-L\|_{L^p(\B)}^n+O_g(t)\le C_{q,d}\|g\|_{L^2(\B)}^2. 
\end{align*}
Taking $g=|x|^2$ and let $t\to 0$ to conclude that $n\ge 2$.

\end{proof}

\begin{proposition}\label{optexp}  Let $d\ge 1$, $p\ge 1$, and  $\overline{q}\ge 4$ an even integer. There exists $\rho=\rho(\overline{q},d)>0$ such that the following holds. If for some $q>3$ with $|q-\overline{q}|<\rho$, there exists $c_{q,d}>0$ such that 

\begin{align*} 
\|\widehat{e^{ig}1_{\B}}\|_q^q &\le    \|\widehat{1_{\B}}\|_q^q-c_{\overline{q},d}\inf_{\substack{L\text{ affine}\\\R-\text{valued}}}  \|e^{i(g-L)}-1\|_{L^p(\B)}^2
\end{align*}
for any function $g:\R^d\to\R$, then $p\le 2$.

\end{proposition}

\begin{proof} Let $g\in L^2(\B)\cap L^\infty(\B)$ be a real-valued function. By the proof of Lemma \ref{optpow}, for sufficiently small $t>0$,

\begin{align*}
\frac{c_{q,d}}{2}\inf_{\substack{L\text{ affine}\\\R-\text{valued}}}  \|tg-L\|_{L^p(\B)}^2 +O_g(t^3)\le C_{q,d}t^2\|g\|_{L^2(\B)}^2. 
\end{align*}
Divide by $t^2$ to get 
\begin{align*}
\frac{c_{q,d}}{2}\inf_{\substack{L\text{ affine}\\\R-\text{valued}}}  \|g-L\|_{L^p(\B)}^2 +O_g(t)\le C_{q,d}\|g\|_{L^2(\B)}^2
\end{align*} 
and let $t\to 0$. Thus 
\[ \frac{c_{q,d}}{2}\inf_{\substack{L\text{ affine}\\\R-\text{valued}}}  \|g-L\|_{L^p(\B)}^2 \le C_{q,d}\|g\|_{L^2(\B)}^2
\]
for real-valued $g\in L^2(\B)\cap L^\infty(\B) $.

Suppose that $p>2$. The real-valued function $g=|x|^{d/p}$ is in $L^2(\B)$ but $\|g\|_{L^p(\B)}=\infty$. Let $g_n:\B\to\R$ denote  the function $g$ if $g\le n$ and $n$ if $g>n$. Since each of the $g_n$ is real valued and in $L^2(\B)\cap L^\infty(\B)$, we have 
\[ \frac{c_{q,d}}{2}\inf_{\substack{L\text{ affine}\\\R-\text{valued}}}  \|g_n-L\|_{L^p(\B)}^2 \le C_{q,d}\|g_n\|_{L^2(\B)}^2.
 \]
Since $g_n$ increase monotonically to $g$, but the dominated convergence theorem, $\lim_{n\to\infty}\|g_n\|_{L^2(\B)}=\|g\|_{L^2(\B)}$. For each $n\in\N$, let $L_n:\R^d\to\R$ be an affine function satisfying 
\[ \|g_n-L_n\|_{L^p(\B)}\le \frac{1}{2}\inf_{\substack{L\text{ affine}\\\R-\text{valued}}}  \|g_n-L\|_{L^p(\B)}. \]

\noindent Then  for $\tilde{C}_{q,d}=C_{q,d}^{1/2}(2/c_{q,d})^{1/2}$

\[ \limsup_{n\to\infty}\|g_n-L_n\|_{L^p(\B)}\le \tilde{C}_{q,d}\|g\|_{L^2(\B)}. \]
If $\limsup_{n\to\infty}\|L_n\|_{L^p(\B)}<\infty$, then 
\[ \infty= \limsup_{n\to\infty}\|g_n\|_{L^p(\B)}-\limsup_{n\to\infty}\|L_n\|_{L^p(\B)}\le \limsup_{n\to\infty}\|g_n-L_n\|_{L^p(\B)}\]
is a contradiction. Now suppose $\limsup_{n\to\infty}\|L_n\|_{L^p(\B)}=\infty$. Then since the $L_n$ are affine, $\limsup_{n\to\infty}\|L_n\|_{L^p(\B\cap\{|x|>1/2\})}=\infty$. Take a subsequence $n_k$ so that $\lim_{k\to\infty}\|L_{n_k}\|_{L^p(\B\cap\{|x|>1/2\})}=\limsup_{n\to\infty}\|L_n\|_{L^p(\B\cap\{|x|>1/2\})}$. Then

\begin{align*}  
\infty&= \limsup_{k\to\infty}\|L_{n_k}\|_{L^p(\B\cap\{|x|>1/2\})}- 2^{d/p}|\B\cap\{|x|>1/2\}|\\
&\le \limsup_{k\to\infty}\|L_{n_k}\|_{L^p(\B\cap\{|x|>1/2\})}- \liminf_{k\to\infty}\|g_{n_k}\|_{L^p(\B\cap\{|x|>1/2\})}\\
&\le \limsup_{n\to\infty}\|g_n-L_n\|_{L^p(\B\cap\{|x|>1/2\})}\le \limsup_{n\to\infty}\|g_n-L_n\|_{L^p(\B)},
\end{align*} 
which is a contradiction. 

\end{proof}

\end{subsection}

\end{document}

%% file: dfmacros.tex
\newcommand{\R}{\mathbb R}

\newcommand{\Z}{\mathbb Z}
\newcommand{\C}{\mathbb C}
\newcommand{\N}{\mathbb N}

\newcommand{\B}{\mathbb B}

\newcommand{\p}{\varphi}
\newcommand{\s}{\psi}

\renewcommand{\a}{\alpha}

\newcommand{\mc}{\mathcal}

\renewcommand{\Re}{\text{Re}}
\renewcommand{\Im}{\text{Im}}